\documentclass[a4paper, 11pt]{amsart}
\usepackage{amsmath, amsthm, amscd, amssymb, amsfonts, amsxtra, amssymb, latexsym}
\usepackage{enumerate}
\usepackage{verbatim}

\usepackage{array}
\usepackage{hyperref}
\hypersetup{colorlinks = true,	allcolors  = blue}

\newcommand{\G}{\Gamma}
\newcommand{\Z}{\mathbb{Z}}

\newcommand{\N}{\mathbb{N}}
\newcommand{\ff}{\mathbb{F}}

\newcommand{\sk}{\smallskip}
\newcommand{\msk}{\medskip}

\newtheorem{thm}{Theorem}[section]
\newtheorem{prop}[thm]{Proposition}
\newtheorem{lem}[thm]{Lemma}
\newtheorem{coro}[thm]{Corollary}

\theoremstyle{definition}
\newtheorem{rem}[thm]{Remark}
\newtheorem{exam}[thm]{Example}

\theoremstyle{remark}

\usepackage{color}

\begin{document} \sloppy
\numberwithin{equation}{section}
\title{Waring numbers over finite commutative local rings} 
\author[R.A.\@ Podest\'a, D.E.\@ Videla]{Ricardo A.\@ Podest\'a, Denis E.\@ Videla}
\dedicatory{\today}
\thanks{Partially supported by CONICET, FONCyT and SECyT-UNC}

\address{Ricardo A.\@ Podest\'a, FaMAF -- CIEM (CONICET), Universidad Nacional de C\'ordoba, \newline
	Av.\@ Medina Allende 2144, Ciudad Universitaria, (5000) C\'ordoba, Argentina. 
	\newline {\it E-mail: podesta@famaf.unc.edu.ar}}
\address{Denis E.\@ Videla, FaMAF -- CIEM (CONICET), Universidad Nacional de C\'ordoba, \newline
	Av.\@ Medina Allende 2144, Ciudad Universitaria,  (5000) C\'ordoba, Argentina. 
	\newline {\it E-mail: devidela@famaf.unc.edu.ar}}

\begin{abstract}
In this paper we study Waring numbers $g_R(k)$ for $(R,\frak m)$ a finite commutative local ring with identity and $k \in \N$ with $(k,|R|)=1$. 
We first relate the Waring number $g_R(k)$ with the diameter of the Cayley graphs 
$G_R(k)=Cay(R,U_R(k))$ and $W_R(k)=Cay(R,S_R(k))$ with $U_R(k)=\{x^k : x\in R^*\}$ and $S_R(k)=\{x^k : x\in R^\times\}$, distinguishing the cases where the graphs are directed or undirected. 
We show that in both cases (directed or undirected), the graph $G_R(k)$ can be obtained by blowing-up the vertices of $G_{\ff_{q}}(k)$ a number $|\frak{m}|$ of times, 
with independence sets the cosets of $\frak{m}$, where $q$ is the size of the residue field $R/\frak m$.
Then, by using the above blowing-up, 
we reduce the study of the Waring number $g_R(k)$ over the local ring $R$ to the computation of the Waring number $g(k,q)$ over the finite residue field $R/\frak m \simeq \ff_q$.
In this way, using known results for Waring numbers over finite fields, we obtain several explicit results for Waring numbers over finite commutative local rings with identity.
\end{abstract}

\maketitle

\section{Introduction}
In this work we study the Waring numbers $g_R(k)$ over finite commutative local rings $R$ with identity, where $(k,|R|)=1$, through the computation of the diameter of certain Cayley graphs with vertex set $R$ and connection set defined by $k$-th powers in $R$. 
The case when $k$ and $|R|$ are not coprime remains open. 
We will reduce the computation of Waring numbers $g_R(k)$ over finite commutative local rings $(R,\frak m)$ to the more known problem of computing the Waring numbers $g(k,q)$ over finite fields $\ff_q$. More precisely, we will show that $g_R(k)$ can be put in terms of $g(k,q)$, where the residue field $R/\frak m$ is isomorphic to $\ff_q$.

\subsection{Preliminaries}
Here we recall the basic facts and definitions of Cayley graphs over rings and Waring numbers over finite rings.

\subsubsection*{Cayley graphs over rings}
Let $G$ be a finite abelian group and $S$ a subset of $G$ with $0\notin S$. The \textit{Cayley graph } $Cay(G,S)$ is the directed graph whose vertex set is $G$ and $v, w \in G$ form a directed edge (or arc) $\overrightarrow{vw}$ of $\Gamma$ from $v$ to $w$ if $w-v \in S$. Since $0\notin S$, $\Gamma$ has no loops. 
Notice that if $S$ is symmetric, that is $-S=S$, then we can consider $Cay(G,S)$ as undirected (and conversely), and hence $Cay(G,S)$ is $|S|$-regular. In fact, if $\overrightarrow{vw}$ is an arc then $\overrightarrow{wv}$ is an arc also, and hence we consider both arcs with opposite directions as a single undirected edge ${vw}$.

One interesting instance of these graphs is when $G$ is a finite field. In particular,
if we let $q=p^m$ with $p$ a prime number and $k$ a non-negative integer with $k\mid q-1$, 
the \textit{generalized Paley graph} (\textit{GP-graph} for short) is the Cayley graph
\begin{equation} \label{Gammas}
	\G(k,q) = Cay(\ff_{q},U_{k}) \qquad \text{with} \qquad U_k = \{ x^{k} : x \in \ff_{q}^*\}
\end{equation} 
where $\ff_q^* = \ff_q \smallsetminus \{0\}$.
Notice that $\G(k,q)$ is an $n$-regular graph with $n=\tfrac{q-1}k$. 
The graph $\G(k,q)$ is undirected either if $q$ is even or if $k$ divides $\tfrac{q-1}2$ when $p$ is odd (equivalently if $n$ is even when $p$ is odd) and it is connected if $n$ is a primitive divisor of $q-1$ (see  \eqref{prim div}).  
When $k=1$ we get the complete graph $\G(1,q)=K_q$ and when $k=2$ we get the classic Paley graph $\Gamma(2,q) = P(q)$. 

GP-graphs have been extensively studied in the few past years. 
Lim and Praeger studied their automorphism groups and characterized all GP-graphs which are Hamming graphs 
(\cite{LP}). 
In \cite{PP}, Pearce and Praeger characterized all GP-graphs which are Cartesian decomposable (some generalization 
to the directed case can be found in \cite{PV7}).
The number of walks in GP-graphs is related with the number of solutions of 
diagonal equations over finite fields (\cite{V}). 
Under some mild restrictions, the spectrum of GP-graphs determines the weight distribution of their associated irreducible
codes (\cite{PV2}, \cite{PV4}). In fact, using this relation between graphs and codes, the spectrum and some spectral properties of the graphs $\G(3,q)$ and $\G(4,q)$ were recently obtained in \cite{PV8}.

Another important special case of Cayley graphs is obtained when $G$ is a finite commutative ring $R$ with identity and $S$ is its group of units $R^*$. That is 
$$G_R = Cay(R,R^*),$$
called the \textit{unitary Cayley graph}.
Unitary Cayley graphs were studied for instance in \cite{Ak+}, \cite{Il}, \cite{Ki+}, \cite{LZ} and \cite{PV5}. 
Recently, Liu and Zhou \cite{LZ2} defined and studied the \textit{quadratic unitary Cayley graphs} 
$\mathcal{G}_R = Cay(R,T_{R})$,
where $T_R=Q_{R}\cup (-Q_R)$ and $Q_{R}=\{x^2: x\in R^*\}$ with $R$ a finite commutative ring with identity. Under some conditions, they showed that in most of the cases these graphs are Kronecker products of classic Paley graphs $P(q)$ and the pseudograph $\mathring{K}_m$ obtained by attaching loops to all of the vertices of a complete graph (see Theorems 2.3 and 2.5 in \cite{LZ2}). This was shown by de Beaudrap \cite{Be} in the case $R=\mathbb{Z}_n$.
This decomposition allowed that author to determine the spectra, energy and other interesting properties of $\mathcal{G}_R$.

\subsubsection*{Waring numbers on finite fields and rings}
The classical problem introduced by Waring near 1770 for the integers $\Z$ can be considered more generally in the context of rings. The Waring's problem on an arbitrary ring $R$ asks, given a natural number $k$, what is the minimum $s\in \mathbb{N}$ 
such that the diagonal equation
	$$x_{1}^{k}+\cdots +x_{s}^{k}=r$$ 
has solutions for all $r\in R$. If this number exists, it is called the \textit{$k$-th Waring number of $R$} and it is denoted by $g_{R}(k)$. Clearly, $g_{\ff_{q}}(k)=g(k,q)$. 
 We are interested in Waring numbers over finite rings.

If $R=\ff_q$ is a finite field of $q$ elements
$g_{\ff_q}(k)$ is usually denoted by $g(k,q)$.
The number $g(k,q)$ not always exist. If $q=p^m$ then $g(k,q)$ exists if and only if $\tfrac{p^m-1}{p^d-1}\nmid k$ for any $d\mid m$ with $d\ne m$. When $g(k,q)$ exists, we have that 
	$$g(k,q)=g((k,q-1),q).$$ 
Since $g(1,q)=1$ trivially, it is customary to assume that $k\mid q-1$.
The study of Waring numbers in finite fields has a long history that can be traced back to Cauchy. After him, several authors studied these numbers. In the last 30 years, we can mention the works of Cipra, Cochrane and Pinner (\cite{Ci}, \cite{CiCP}, \cite{CP}), Garc\'ia-Sol\'e (\cite{GS}), Glibichuk et al.\@ (\cite{Gl}, \cite{GlR}), Kononen (\cite{KK}), Konyagin (\cite{Kon}), Moreno-Castro (\cite{MC1}, \cite{MC}), Winterhof et al.\@ (\cite{Win}, \cite{Win1}, \cite{Win2}) and recently Podestá-Videla (\cite{PV6}, \cite{PV7}).	
The connection between Waring numbers over arbitrary finite fields $\ff_q$ and GP-graphs was established in \cite{PV6} following ideas of Hamidoune and García-Solé (see \cite{GS}) in the case $\ff_p$ with $p$ prime.
Namely, if the GP-graph $\G(k,q)$ is connected (not necessarily undirected), 
then we have (see Theorem 3.3 in \cite{PV6})
\begin{equation} \label{g=d} 
	g(k,q) = \delta(\G(k,q)),
\end{equation}
where $\delta(\G)$ denotes the diameter of a graph $\G$.

The Waring problem over finite rings has received much less attention than over finite fields. 
It was probably initiated by Charles Small in 1977 studying the problem for $\Z_n$ (\cite{Sm}, \cite{Smb}, \cite{Smc}), the so called Waring's problem modulo $n$.
There are few works on these topics, even in the commutative case, although during the last 5 years some works have appeared.
In the non-commutative case, we mention Karabulut's paper \cite{K} where the asymptotic Waring problem over arbitrary finite rings is studied. Also, very recently, Kishore (\cite{K2}) and Kishore-Singh (\cite{K2b}) studied the Waring numbers over matrix rings over finite fields $M_n(\ff_q)$ improving the asymptotic results of Karabulut for the rings of matrices over finite fields. 
Namely, he proved Larsen's conjecture, i.e.\@ that for any $k\ge 1$, there is a constant $C_k$ depending only on $k$ such that for any $q>C_k$ and $n\ge 1$ every matrix in $M_n(\ff_q)$ can be written as the sum of two $k$-th powers in $M_n(\ff_q)$.

\subsection{New families of Cayley graphs}
Here we define two new families of Cayley graphs over rings whose connection sets are the $k$-th powers of units and $k$-th powers of non-zero elements, respectively. In what follows, let $k\in\mathbb{N}$ and let $R$ be a finite commutative ring (not necessarily local) with identity and denote by $R^*$ the units of $R$ and by $R^\times = R \smallsetminus \{0\}$. 

We define the \textit{unitary Cayley graph of $k$-th powers} or \textit{$k$-unitary Cayley graph} by 
\begin{equation} \label{GRk}
	G_R(k)= Cay(R,U_R(k)) \qquad \text{where} \qquad U_R(k) = \{x^k: x\in R^*\}.
\end{equation} 
Clearly, $G_R(1)$ is the unitary Cayley graph $G_R=Cay(R,R^*)$. 
In general, $G_{R}(k)$ is directed.
Notice that $U_R(k)$ is a symmetric set if and only if $-1\in U_R(k)$. 
They are generalizations of generalized Paley graphs over finite fields to commutative rings with identity. 
In particular, when $R$ is a finite field $\ff_q$ of cardinality $q=p^n$ then $G_R(k)$ is the generalized Paley graph
\begin{equation} \label{GPkq}
	G_{\ff_q}(k) = \G(k,q) 
\end{equation}
defined in \eqref{Gammas}.
See \cite{PV2} (also \cite{PV4}, \cite{PV5}) for more information and results on the graphs $\G(k,q)$.

We also define the \textit{Waring graph of $k$-th powers over $R$} as
\begin{equation} \label{WRk}
	W_R(k) = Cay(R, S_R(k)) \qquad \text{where} \qquad S_R(k) = \{x^k: x\in R^{\times}\}.
\end{equation}
The difference with $G_R(k)$ is that now powers of nilpotent elements are allowed in the connection set. 
In general, $G_{R}(k)$ and $W_{R}(k)$ are directed graphs. 
The graph $W_{R}(k)$ may have loops (for instance if $R$ has nilpotent elements of degree less than $k$). 
In any case, since $U_R(k) \subset S_R(k)$, in fact $U_R(k) = S_R(k) \cap R^*$, we have that $G_{R}(k)$ is a subgraph of $W_{R}(k)$. Furthermore, $G_{R}(k)$ is undirected if and only if $W_{R}(k)$ is undirected as well, since if $-1 \in R^*$ we have 
$-1 \in U_R(k)$ if and only if $-1\in S_R(k)$.
Notice that, if $G_{R}(k)$ is connected, then $W_{R}(k)$ is also connected.
When $R=\ff_q$ is a finite field of cardinality $q=p^n$ with $p$ prime, then 
\begin{equation} \label{W=G para R=Fq}
	W_{\ff_q}(k) = G_{\ff_q}(k).  
\end{equation}

\subsection{Outline and main results}
In what follows let $R$ be a finite commutative ring with identity and let $k\in \N$ be coprime with $|R|$. We now give a brief account of the main results in the paper. 
 
In Section \ref{sec2} we study the structure of the graph $G_R(k)$ for $R$ a local ring with unique maximal ideal $\frak m$.
In Theorem \ref{Local case} we give the Kronecker product decomposition of $G_R(k)$, namely $G_R(k) \simeq G_{\ff_q}(k) \otimes \mathring{K}_{m}$, where $\mathring{K}_{m}$ is the complete graph of $m$ vertices with a loop added at every vertex and $m=|\frak m|$.
Also, we show that $G_{R}(k)$ is a balanced blow-up of order $m$, more precisely $G_R(k) \simeq {\G(k,q)}^{(m)}$, whose independent sets are all the cosets of $\frak m$ in $R$. In Corollary~\ref{-1RFq} we give arithmetic conditions for these graphs to be undirected or connected.

In Section \ref{sec3} we study the relation between the Waring number $g_R(k)$ and the diameter of the graphs $G_R(k)$ and $W_R(k)$, with $R$ not necessarily local.
In Theorem 2.1 of \cite{PV6} we showed that the Waring number $g(k,q)$ over the finite field $\ff_q$ equals the diameter of the generalized Paley graph $\G(k,q)$. In Theorem \ref{eqbound} we extend this result to any (not necessarily finite) commutative ring with identity.
In Proposition \ref{lem1} we compute the diameter $\delta (\G\otimes \mathring{K}_m)$ for any graph $\G$, distinguishing the cases when $\G$ is directed or undirected. In the directed case, the directed girth of a graph plays a role.  	
	
In Section \ref{sec4} we study Waring numbers over a finite commutative local ring with identity $(R,\frak m)$. 
In Theorem \ref{teo waring Local case}, one of the main results, we show that if 
$(k,|R|)=1$ then, under certain arithmetic conditions on $k$ and $q$, the number $g_R(k)$ exists and in this case it can be put in terms of $g(k,q)$, where $q$ is the size of the residue field, that is $R/\frak m \simeq \ff_q$. In fact, $g_R(k)$ can take the values $1,2$ or $g(k,q)$ if $G_R(k)$ is undirected and $g(k,q)$ or $g(k,q)+1$ if $G_R(k)$ is directed (see \eqref{gR local} and \eqref{gR local dir}). 

As a consequence, in Section \ref{sec5red} we give a reduction formula for Waring numbers over finite local rings (as in \cite{PV7} for finite fields) . In fact, we show in Theorem \ref{teo reduction} that under certain mild arithmetic conditions on $a,b,c,p \in \N$, with $p$ prime, 
the Waring numbers $g_{R_{ab}}(\tfrac{p^{ab}-1}{bc})$ and $g_{R_a}(\tfrac{p^{a}-1}{c})$ exist. In this case, reduction formulas of the kind
\begin{equation*} \label{red fla intro}
	g_{R_{ab}}(\tfrac{p^{ab}-1}{bc}) = b g_{R_a}(\tfrac{p^a-1}c) 
		\qquad \text{or} \qquad 
	g_{R_{ab}}(\tfrac{p^{ab}-1}{bc}) = b (g_{R_a}(\tfrac{p^a-1}c)-1)
\end{equation*} 
hold for finite commutative local rings with identity $R_{ab}$ and $R_{a}$ whose corresponding residue fields have sizes $p^{ab}$ and $p^a$, respectively.

In Section \ref{sec5} we give explicit formulas for Waring numbers over finite commutative local rings.
We first extend Kononen's formula for Waring numbers over finite fields to Waring numbers over finite commutative local rings $R$ whose residue field $\ff_q$ has size $q=p^{\varphi(r^m)}$ with $p$ prime, namely 
$$g_R \big( \tfrac{p^{\varphi(r^m)}-1}{r^m} \big) = \tfrac 12 (p-1) \varphi(r^m)$$ 
(see Proposition \ref{kononen} for details). In Theorem \ref{prop gral fla b} we give conditions for the formula 
$$ g_R(\tfrac{p^{ab}-1}{b(p^a-1)}) = b$$
to hold for $R$ a finite commutative local ring and $a,b,p \in \N$ with $p$ prime. Using this last result, in Example \ref{coros 6.3-6.5} we get infinite families of integers $k_{p,a}$ depending on $p$ and $a$, such that 
$$g_{R_{ta}}(k_{p,a}) \in \{2,3,5,7\}$$
where $R_{ta}$ is any finite commutative local ring with residue field of size $p^{ta}$, for certain $t \in \N$.

Finally, in Section \ref{sec6} we apply previous results to get explicit Waring numbers over finite commutative local rings with identity. First, in Theorem \ref{gR2} we give conditions to obtain small Waring numbers over local rings $R$, namely 
$$2 \le g_R(k) \le 3.$$ 
Thus, using Small's result for Waring numbers over finite fields, we get explicit small Waring numbers over finite commutative local rings in Proposition \ref{coro gR23}. Then, considering the case in which $\G(k,q)$ is a connected strongly regular graph, in Proposition \ref{prop gR2} we show that $g_R(2)=2$ for finite commutative local rings $R$ whose residue field is $\ff_q$. 
Lastly, we consider the case when $(R,\frak m)$ is a local ring with prime residue field $\Z_p$, like for example $\Z_{p^s}$ or  $\Z_p[x]/(x^s)$. In this case, we obtain Proposition \ref{prop Zps} as a version of Theorem \ref{gR2} with one less hypothesis. 
Also, for any $p$ odd, in Corollary \ref{coro exact Zp} we obtain the values $g_R(2)=2,3$ for $p\equiv 1,3 \pmod 4$ respectively, $g_R(\frac{p-1}2)=2$ and $g_R(p-1)=p$, while in Corollary \ref{coro p2} we also get $g_R(\frac{p+1}2)=\frac{p+1}2$ for any $R$ of size $p^2$.

\section{The structure of $G_R(k)$ for $R$ local} \label{sec2}
In this section we study some structural properties of the graphs $G_R(k)$ defined in \eqref{GRk}, i.e. 
$$G_R(k)=Cay(R,U_R(k)) \qquad \text{with} \qquad U_R(k)=\{x^k : x \in R^*\},$$ 
for $R$ a finite commutative local ring $(R,\frak m)$ with identity, where $k$ is coprime with $|R|$. We show that 
$G_R(k)$ can be decomposed as Kronecker products, which are actually balanced blow-ups, and we give arithmetic conditions for (non)directedness and connectedness of the graphs.

We begin by showing that the function $x \mapsto x^k$ on $R/\frak m$ is a bijection from the coset 
$a+\frak m$ to the coset $a^k+\frak m$ for any $a \in R^*$. Throughout the paper we will denote by $(a,b)$ the greatest common divisor of integers $a$ and $b$.

\begin{lem} \label{prop local}
	Let $(R,\frak{m})$ be a finite commutative local ring with identity
	with associated residue field $R/\frak{m} \simeq \ff_q$. 
	Let $k\in \mathbb{N}$ be such that $(k,|R|)=1$.
	Then, we have: 
	\begin{enumerate}[$(a)$]
		\item For every $a\in R^*$ the function $g_a :  a+\frak{m} \rightarrow a^{k}+\frak{m}$ given by $g_a(x)=x^k$ is a bijection. \sk
		
		\item Let $a,b\in R$ be such that $b-a\in U_R(k)$. 
		If $a\equiv c\pmod{\frak{m}}$ and $b\equiv d\pmod{\frak{m}}$, then we have $d-c\in U_R(k)$. 
	\end{enumerate}
\end{lem} 

\begin{proof}
($a$) Notice that since $a+\frak m$ and $a^{k}+\frak m$ have the same number of elements, by finiteness of $R$ it is enough to show that $g_a$ is injective.
Assume the opposite, i.e.\@ suppose that there are elements $a+m_1, a+m_2$ in $a+\frak m$ with $m_1\neq m_2$ such that 
	$(a+m_1)^k=(a+m_2)^k$. 
Since $a\in R^{*}$ we get $a+m_2\in R^{*}$ as well and so, from the above expression, we obtain that 
	$$((a+m_1)(a+m_2)^{-1})^{k}=1.$$ 
Moreover, we have that $(a+m_1)(a+m_2)^{-1}\in 1+\frak{m}$, since 
in general $(b+\frak{m})(c+\frak{m})=bc+\frak{m}$ and $(a+\frak{m})^{-1}=a^{-1}+\frak{m}$ 
for $a\in R^*$ and $b,c\in R$.
Notice that $(a+m_1)(a+m_2)^{-1}= 1+m$ for some $m\in \frak{m}\smallsetminus \{0\}$ since $m_1\neq m_2$. Hence, 
we obtain a $k$-th root of unity different from $1$ in the coset $1+\frak{m}$.
	
Now, since $m\in \frak{m}\smallsetminus \{0\}$ there exists a minimum $N>1$ such that $m^{N}=0$.
On the other hand, since $(1+m)^{k}=1$ we obtain that
	\begin{equation*}
		km + \sum_{2 \le \ell \le k} \tbinom{k}{\ell}m^{\ell} = 0.
	\end{equation*}
Hence, by multiplying the above identity by $m^{N-2}$ we have that
$km^{N-1}=0$. Thus, since $k$ is coprime with the characteristic of $R$, we obtain that $k \cdot 1_{R}$ is a unit in $R$. 
Then, we have that $m^{N-1}=0$, which contradicts the minimality of $N$. 
Therefore $g$ is injective and, thus, it is a bijection as asserted.
	
	\msk 
	
\noindent ($b$) 
Since $b-a\in U_R(k)$, then there exists $e\in R^{*}$ such that $b-a=e^{k}$. 
On the other hand, by hypothesis $a\equiv c\pmod{\frak{m}}$ and $b\equiv d\pmod{\frak{m}}$, these imply that
	$$(d-c)+\frak{m}=(b-a)+\frak{m}= e^{k}+\frak{m},$$
in particular $(d-c)\in e^{k}+\frak{m}$. Thus, by ($a$), we obtain that there exists $r \in e+\frak{m}$ such that $r^{k}=d-c$.
Finally, since $e\in R^{*}$ then $r\in R^{*}$ and therefore $d-c\in U_R(k)$, as asserted.
\end{proof}

The previous result for cosets on a local ring $R$ directly implies the following result for vertices and arcs on the graph $G_R(k)$. 
\begin{prop} \label{Thm mod Local case}
Let $(R,\frak{m})$ be a finite commutative local ring with identity and let $k\in \mathbb{N}$ be such that $(k,|R|)=1$. 
Denote by $E$ the set of arcs in the graph $G_R(k)$.
If $ab \in E$ then $cd \in E$ for every $c\in a+\frak m$ and $d\in b+\frak m$. The vertices in each coset $a+\frak m$ of $\frak m$ are independent. \hfill $\square$
\end{prop}

We now give the Kronecker product decomposition of $G_R(k)$ in terms of GP-graphs.
We recall that the \textit{Kronecker product} of the graphs $\G_1,\ldots, \G_s$, denoted by 
$$\G=\G_1\otimes \cdots\otimes \G_s,$$ 
is the graph with vertex set $V(\Gamma)=V(\G_1) \times \cdots \times V(\G_s)$ 
where two vertices $v=(v_1,\ldots,v_s)$ and $w=(w_1,\ldots,w_s)$ 
are neighbors in $\Gamma$ if and only if $v_i$ and $w_i$ are neighbors in the graph $\G_i$ for all $i=1,\ldots,s$.
This product is associative and commutative. 
The Kronecker product extends naturally to pseudographs (directed graphs or graphs with loops). 
In this case, notice that $\G_1 \otimes \G_2$ is loopless if $\G_1$ or $\G_2$ is loopless, and $\G_1\otimes \G_2$ is directed if $\G_1$ or $\G_2$ is directed. 

For any graph $G$ and $t\in \N$, the \textit{(balanced) blow-up of order $t$} of $G$, denoted  $G^{(t)}$, is the graph obtained by replacing each vertex $x$ of $G$ by a set $V_x$ of $t$ independent vertices and every edge $\{x, y\}$ of $G$ by a complete bipartite graph $K_{t,t}$ with parts $V_x$ and $V_y$ (of course $G^{(1)} = G$). 
Notice that we have the natural isomorphism 
\begin{equation} \label{blowup}
G \otimes \mathring{K}_m \simeq G^{(m)},
\end{equation}
where $\mathring{\G}$ denotes a graph $\G$ with a loop added at each vertex.

\begin{thm} \label{Local case}
Let $(R,\frak{m})$ be a finite commutative local ring with $m=|\frak{m}|$ and residue field $R/\frak{m} \simeq \ff_q$. If $k\in \mathbb{N}$ satisfies $(k,|R|)=1$, then
\begin{equation} \label{Local form G curv}
G_R(k) \simeq {\G(k,q)}^{(m)} \simeq G_{\ff_q}(k) \otimes \mathring{K}_{m} ,  
\end{equation} 
where ${\G(k,q)}^{(m)}$ denotes the balanced blow-up of order $m$ of $\G(k,q)$, whose independent sets are all the cosets of $\frak m$ in $R$, and $\mathring{K}_{m}$ is the complete graph of $m$ vertices with a loop added at every vertex.
In particular, $G_R(k)$ is $\frac{m(q-1)}{k}$-regular.
\end{thm}

\begin{proof}
	Clearly, from Proposition \ref{Thm mod Local case}, 
	we have that $G_{R}(k)$ is a balanced blow-up of order $m$ of $\G(k,q)$, where the independent sets are all of the cosets of $\frak m$ in $R$. 
	That is
	$$G_R(k) \simeq {\G(k,q)}^{(m)}.$$
	Now, by \eqref{blowup} we have that ${\G(k,q)}^{(m)}\simeq G_{\ff_q}(k) \otimes \mathring{K}_{m}$, and thus we obtain \eqref{Local form G curv} as desired.
	The last assertion is clear from \eqref{Local form G curv} and the fact that $\G(k,q)$ is $\frac{q-1}{k}$-regular (see the Introduction).
\end{proof}

\begin{exam}
Consider the local ring $\Z_n$ of integers modulo $n$, where $n=p^r$ with $p$ prime. Then, the maximal ideal is $\frak m= \Z_{p^{r-1}}$ and, thus, we have 
$$G_{\Z_{p^r}}(k) \simeq G_{\Z_p}(k) \otimes \mathring{K}_{p^{r-1}} \simeq \G(k,p)^{(p^{r-1})}$$
for any $k$ coprime to $n$. \hfill $\lozenge$
\end{exam}

In the next result we obtain some structural consequences for the graphs and the rings.
We give two arithmetic conditions: the first one 
relates the belonging of $-1$ to $U_R$ and $U_{R/\frak m}$ for a local ring $(R,\frak m)$, 
hence it has to do with non-directness of the graphs, while the second one is on connectivity of the graphs. 

We first need to recall that an integer $n$ is a \textit{primitive divisor} of $p^m-1$ if $n\mid p^m-1$ and $n\nmid p^{a}-1$ for all $1\le a<m$. 
For simplicity, as in our previous works \cite{PV7} we denote this fact by  
	\begin{equation} \label{prim div}
		n\dagger p^m-1.
	\end{equation} 
Also, it is well-known that 
	\begin{equation} \label{con cond}
	\G(k,q) \text{ is connected} \quad \Leftrightarrow \quad n \dagger q-1,
	\end{equation}
where $n$ is the regularity degree of $\G(k,q)$, that is 
	$$n=\tfrac{q-1}{k'} \quad \text{ where } \quad k'=(k,q-1).$$

\begin{coro} \label{-1RFq}
Let $(R,\frak{m})$ be a finite commutative local ring with $m=|\frak{m}|$ and residue field $R/\frak{m} \simeq \ff_q$. 
Let $k\in \mathbb{N}$ such that $(k,|R|)=1$. 
Then, 

\begin{enumerate}[$(a)$]
	\item $-1\in U_R(k)$ if and only if $-1\in U_{R/\frak m,k}$. In particular, $-1 \in U_R(k)$ if and only if $q$ is even or else if $q$ is odd and $(k,q-1) \mid \frac{q-1}{2}$. \sk
	
	\item $G_R(k)$ is undirected if and only if $G_{R/\frak m}(k)$ is undirected. \sk 

\nopagebreak 	
	\item $G_{R}(k)$ is connected if and only if $\frac{q-1}{(k,q-1)} \dagger q-1$.
\end{enumerate}
\end{coro}

\begin{proof}
($a$)
Clearly, if $-1\in U_R(k)$ then by projecting on $R/\frak m$, we obtain that $-1 \in U_{\ff_{q},k}$ as asserted.
	
Conversely, if $-1 \in U_{\ff_{q},k}$, then $G_{\ff_{q}}(k)$ is undirected. 
On the other hand, since $(k,|R|)=1$ by Theorem \ref{Local case}, we have that $G_{R}(k) = G_{\ff_{q}}(k) \otimes \mathring{K}_m$. 
Thus, the graph $G_{R}(k)$ is undirected since the Kronecker product is closed in the family of undirected graphs.
Hence, we have that the set $U_R(k)$ is symmetric and so $-1 \in U_R(k)$.
	
The last statement follows from the fact that $-1 \in U_{\ff_{q},k}$ if and only if $q$ is even or else if $q$ odd and $(k,q-1)\mid \frac{q-1}{2}$.

\noindent ($b$) 
This is a direct consequence of ($a$).

\noindent ($c$) 
Assume first that $\frac{q-1}{(k,q-1)} \dagger q-1$. Note that $(k,q)=1$ since $(k,|R|)=1$ and so
by Theorem~\ref{Local case}, we have that
$G_{R}(k) \simeq  G_{\ff_q}(k) \otimes \mathring{K}_{m}$, where $G_{\ff_q}(k)=\G(k,q)$ by \eqref{GPkq}.
The graph $\G(k,q)$ is connected since $\frac{q-1}{(k,q-1)} \dagger q-1$, by \eqref{con cond}. Thus, $\G(k,q) \otimes \mathring{K}_{m}$ is connected as well, since the Kronecker product between a  connected graph (digraph) and $\mathring{K}_m$ is always connected. Therefore, $G_{R}(k)$ is connected as asserted. 

Conversely, assume that $G_R(k)$ is connected.
If $\frac{q-1}{k'}$ is not a primitive divisor of $q-1$, then $G_{\ff_q}(k)$ is not connected and, thus, its trivial eigenvalue  $\lambda_1(G_{\ff_q}(k)) = \frac{q-1}{(k,q-1)}$ has multiplicity greater than $1$.
Since $G_{R}(k) \simeq  G_{\ff_q}(k) \otimes \mathring{K}_{m}$, by Theorem \ref{Local case}, the trivial eigenvalue $$\lambda_1(G_R(k)) = m\lambda_1(G_{\ff_q}(k)) = \tfrac{m(q-1)}{(k,q-1)}$$ 
of $G_{R}(k)$ has multiplicity greater than $1$ and hence $G_R(k)$ is not connected. 
\end{proof}

It is known (for instance by using Hensel's lemma and the fact that a finite ring is complete) that the number of $k$-th roots of unity in $(R,\frak m)$ is the same as the number of $k$-th roots of unity in $R/\frak m$, for any $k$ coprime with the size of $R$. 
We finish the section with a direct graph theoretical proof of this fact.

\begin{prop} \label{k-th roots of 1}
Let $(R,\frak{m})$ be a finite commutative local ring with identity and let $k\in \mathbb{N}$ be such that $(k,|R|)=1$. 
Then, $R$ and $R/\frak m$ have the same number of $k$-th roots of unity.
\end{prop}

\begin{proof}
Let $u_{R}(k)=|\{x\in R^{*}: x^{k}=1 \}|$ be the number of $k$-th roots of the unity on $R$ and  
put $u_R(k,a) =|\{x\in R^{*}: x^{k}=a\}|$ for $a\in R^*$. Then, we have $u_R(k,1)=u_R(k)$ and
$$\sum_{a \in U_R(k)} u_R(k,a) = |R^*|.$$
Now, fix $x\in R$. For any $y\in R^*$ such that $x^k=y^k =a$ we have that $(xy^{-1})^k=1$, and this implies that
$u_R(k,a)=u_{R}(k)$ for all $a\in U_R(k)$. Hence 
$$u_{R}(k) |U_R(k)|=|R^*|.$$
In the same way, we obtain that
$$u_{R/\frak{m}}(k) |U_{R/\frak{m},k}| = |(R/\frak{m})^*|.$$

On the other hand, we have that $|U_R(k)| = m|U_{R/\frak{m},k}|$ where $m=|\frak{m}|$, since the degree of regularity of $G_R(k)$ is $m$ times the regularity degree of $G_{R/\frak m}(k)$ by Theorem \ref{Local case}.
Finally, since $|R^*|=m|(R/\frak{m})^*|$, we finally obtain that 
$$u_{R/\frak{m}}(k) = \frac{|(R/\frak{m})^*|}{|U_{R/\frak{m},k}|} = \frac{|R^*|}{|U_R(k)|}=u_{R}(k),$$
as asserted.
\end{proof}

\section{The diameter of $W_R(k)$ and Waring numbers $g_R(k)$} \label{sec3}
In Theorem 2.1 of \cite{PV6} we showed that the Waring number $g(k,q)$ over finite fields $\ff_q$ equals the diameter of the generalized Paley graph $\G(k,q)$, provided that this graph is connected (so that the Waring number exists). This kind of relationship can be extended to (not necessarily finite) commutative rings $R$ with identity as well. Namely, the Waring number $g_{R}(k)$ is related to the diameters of the graphs $G_{R}(k)$ and $W_{R}(k)$ as we next show.

\subsection{Waring numbers as diameters of Cayley graphs.}
We will denote the diameter of a graph $G=G(V,E)$ by $\delta(G)$.  
We recall that the diameter of $G$ is the greatest distance between any pair of vertices, that is
$$ \delta(G) = \max _{u,v\in V} d(v,u)$$
where $d(u,v)$ denotes the distance between the vertices $u$ and $v$, that is the number of (directed) edges in a shortest path (or geodesic) between $u$ and $v$.

As we have done in \cite{PV6} for finite fields, we now relate the diameter of the GP-graphs $W_R(k)$ with the Waring numbers $g_R(k)$, for a finite commutative ring $R$ with identity.

\begin{thm} \label{eqbound}
	Let $R$ be a commutative ring with identity and let $k\in \N$. 
	Then, we have:
	\begin{enumerate}[$(a)$]
		\item If $W_{R}(k)$ is connected with $\delta(W_{R}(k))<\infty$, then $g_{R}(k)$ exists and $g_{R}(k)=\delta (W_{R}(k))$. \sk
	
		\item If $G_{R}(k)$ is connected with $\delta(G_{R}(k))<\infty$, then $W_{R}(k)$ is also connected, $g_{R}(k)$ exists and $$g_{R}(k) = \delta (W_{R}(k)) \le \delta (G_{R}(k)).$$	
	\end{enumerate}
\end{thm}

\begin{proof}
$(a)$ We proceed similarly as in Theorem 2.1 in \cite{PV6}, but we include the proof for completeness.
First, we show that the diameter can be realized with paths starting from $0$.
Let $u,v$ be vertices of $W_R(k)$ such that 
$$d(u,v)=\delta(W_{R}(k))=\delta.$$ 
Clearly, $\delta \ge d(0,c)$ for all $c\in R^{\times}$. In particular, we have that $\delta\ge d(v-u,0)$. 
Assume that $\delta > d(v-u,0)=t$. Then, there is a sequence $x_{1}, \ldots, x_{t} \in R^{\times}$ such that 
$$v-u = x_1^k + \cdots + x_t^k.$$ 
This induces a walk from $u$ to $v$, i.e.\@ $d(u,v) \le t < \delta = d(u,v)$, which is absurd. Therefore $\delta = d(0,v-u)$, and we have that 
\begin{equation}\label{diam gama}¨
	\delta(W_{R}(k)) = \max_{c\in R^{\times}} d(0,c).
\end{equation}

Notice that every element of $R^{\times}$ can be written as a sum of $\delta$ $k$-th powers. In fact, if 
$c\in R^{\times}$ with $d(0,c) = s'$ then $s' \le \delta$, then there exist $y_{1},\ldots,y_{s'}$ such that 
$$y_1^k + \cdots + y_{s'}^k = c.$$ 
Defining $y_{s'+1} = \cdots = y_{\delta}=0$ 
we obtain that $y_1^k + \cdots + y_{\delta}^k = c$, as desired.

Now, by \eqref{diam gama}, there exists $a\in R^{\times}$ such that $\delta(\G)=d(0,a)$. 
Clearly, $a$ cannot be written as a sum of less than a number $\delta$ of $k$-th powers, otherwise we would obtain a walk from $0$ to $a$
in $W_{R}(k)$ of length less than $d(0,a)$. Therefore, $g_{R}(k)$ exists and $g_{R}(k) = \delta(W_{R}(k))$, as claimed.

\sk 
\noindent $(b)$ Notice that since $G_{R}(k)$ is connected and it is a subgraph of $W_{R}(k)$ with the same vertex set, then $W_{R}(k)$ is connected and 
we have that $\delta(W_{R}(k))\le \delta(G_{R}(k))$, and so item ($a$) implies the result.
\end{proof}

\begin{rem}
If we take $R=\ff_{q}$ a finite field we re-obtain Theorem 2.1 of \cite{PV6}. 
\end{rem}

\subsection{Relation between girth and diameter}
We recall that the \textit{girth} of a (directed) graph $G$ is the length of the shortest (directed) cycle in the graph if $G$ has (directed) cycles, otherwise it is infinite. It is usually denoted by $g(G)$ but we will denote it by $\gamma(G)$ to avoid confusion with the Waring numbers $g_R(k)$ and $g(k,q)$. That is, we have 
$$\gamma(G) = \begin{cases} 
	\min\limits_{C_n \hookrightarrow G} n, & \qquad \text{if $G$ has (directed) cycles}, \\[1mm]
	\hfil \infty, & \qquad \text{if $G$ is acyclic}.     
\end{cases}$$

For the next result we will need this lemma relating the girth and the diameter of a directed connected graph.
\begin{lem} \label{g=d+1}
	If $\G$ is a directed connected graph and $\gamma(\G)> \delta(\G)$ then $\gamma(\G)=\delta(\G)+1$.
\end{lem}

\begin{proof}
	Let $v_0 v_1 \cdots v_\gamma$ be a closed directed walk in $\G$ which realizes the girth, i.e.\@ $v_0=v_\gamma$ and 
	$\gamma(\G)=\gamma$.
	Notice that, there is an arrow from $v_{\gamma-1}$ to $v_0$, and so $d(v_0,v_{\gamma-1})=\gamma-1$, otherwise we would obtain a cycle of length less than $\gamma(\G)$.
	Hence, by maximality of the diameter we obtain that $\gamma(\G)-1\le \delta(\Gamma)$. Therefore, $\gamma(\G)=\delta(\G)-1$, as claimed.
\end{proof}

The following result will be useful to understand the diameter of $G_{R}(k)$ in the next section.
From now on we will need to denote by $d_\G$ the distance in a graph $\G$ to consider and compare distances between different graphs simultaneously.
\begin{prop} \label{lem1}
	Let $\G$ be a  graph and let $m,n$ be positive integers. 
\begin{enumerate}[$(a)$]
	\item If $\G$ is undirected and connected then $\G\otimes \mathring{K}_m$ is also undirected and connected and 
	\begin{equation} \label{diam1}
	\delta (\G\otimes \mathring{K}_m) = \begin{cases}
	\delta(\G),  & \qquad \text{if $\G \ne K_n$ for any $n\in \N$}, \\[1mm]
	\hfil 2,	 & \qquad \text{if $\G=K_n$ for some $n\ge 2$ and $m\ge 2$}, \\[1mm]
	\hfil 1,	 & \qquad \text{if $\G=K_1$ or else $\G=K_n$ and $m=1$}.
	\end{cases}
	\end{equation}
	
	\item If $\G$ is directed arc-transitive and connected, then $\G\otimes \mathring{K}_m$ is directed and connected and 
	\begin{equation} \label{diam2}
		\delta (\G\otimes \mathring{K}_m)=
		\begin{cases}
		\hfil \delta(\G), 	& \qquad \text{if $\gamma(\G)\le \delta(\G)$ or $m=1$},	\\[1mm]
		\hfil \delta(\G)+1, & \qquad \text{if $\gamma(\G)> \delta(\G)$ and $m\ge 2$}.
		\end{cases}
	\end{equation}
\end{enumerate}
\end{prop} 

\begin{proof}
($a$) Let us see first that $\G\otimes \mathring{K}_m$ is connected. Let $(v,i),(w,j)$ be any pair of vertices of $\G\otimes \mathring{K}_m$. 
Since $\G$ is connected there is a walk $v v_1 \cdots v_m$ from $v$ to $w$ in $\G$, where $v_m = w$. Thus,  $(v,i)(v_1,j) \cdots(v_m,j)$ is a walk in $\G \otimes \mathring{K}_m$
from $(v,i)$ to $(w,j)$. Also, since $\mathring{K}_m$ and $\G$ are undirected graphs ($\G$ by hypothesis), their Kronecker product is also undirected.

Now we compute the diameter of $\G\otimes \mathring{K}_m$. Assume first that $\G$ is not a complete graph.
To see that $\delta (\G\otimes \mathring{K}_m)=\delta(\G)$, it is enough to show that given $v,w \in V(\G)$ we have
	$$d((v,i),(w,j))=d(v,w)$$
for any $i,j\in \{1,\ldots,m\}$.
	For this purpose, notice that by the definition of the Kronecker product for all $i,j \in \{1,\ldots,m\}$ we have 
	\begin{equation}\label{neighbors}
		(v,i) \sim_{\G\otimes \mathring{K}_m} (w,j)\qquad  \Leftrightarrow \qquad v\sim_{\G} w. 
	\end{equation}
	Clearly, \eqref{neighbors} implies the above assertion. 
		Thus, if $(v,i),(w,j) \in V(\G)\times \{1,\ldots,m\}$ satisfy   
	$d((v,i),(w,j)) = \delta (\G\otimes \mathring{K}_m)$, then by our first assertion we have that
	$$\delta (\G\otimes \mathring{K}_m) = d((v,i),(w,j)) = d(v,w).$$ 
	By definition of diameter, we obtain that $\delta(\G)\ge \delta (\G\otimes \mathring{K}_m)$. In a similar manner, we can obtain that
	$\delta(\G)\le \delta (\G\otimes \mathring{K}_m)$, and therefore $\delta (\G\otimes \mathring{K}_m)=\delta(\G)$, as desired. 

 Now, suppose that $\G=K_n$ for some $n\ge 1$. Then, $\G \otimes \mathring{K}_m$ is an $n$-partite complete graph with partitions of size $m$. Thus, assuming $n\ge 2$ we have that
$\delta( \G \otimes \mathring{K}_m)=2$ for $m\ge 2$, as asserted (the case $m=1$ is trivial).

\noindent ($b$) 
That $\G \otimes \mathring{K}_m$ is connected in this case can be proved in the same way as in item ($a$).
On the other hand, 
it is well-known that if either $G$ or $H$ is a directed graph, 
then $G \otimes H$ is also directed. Hence, $\G \otimes \mathring{K}$ is directed, since $\G$ is directed by assumption.

Clearly, if $m=1$, then $\G\otimes \mathring{K}_1=\G$ and so $\delta(\G\otimes \mathring{K}_1)=\delta(\G)$, as asserted. 
Now, assume that $m \ge 2$ and let $(v,i)$ and $(w,j)$ be two vertices in $\G\otimes \mathring{K}_m$ which realize the diameter, that is $\delta(\G \otimes \mathring{K}_m)=d_{\G\otimes \mathring{K}_m}((v,i),(w,j))$. 
Thus, there are two cases to consider, either $v\neq w$ or else $v=w$. 

If $v\neq w$, then it can be shown that $d_{\G \otimes \mathring{K}_m}((v,i),(w,j)) = d_{\G}(v,w)$ as in ($a$) and so in this case 
$\delta(\G \otimes \mathring{K}_m)=\delta(\G)$. 
Notice that in this case a directed cycle which begins and ends in $v$ induces 
a directed walk from $(v,i)$ to $(v,j)$ with $i\neq j$. In particular a directed cycle of length $\gamma(\G)$ induces a minimal directed walk from $(v,i)$ to $(v,j)$, that is $d_{\G \otimes \mathring{K}_m}((v,i),(v,j))=\gamma(\G)$ and so 
we have that $\gamma(\G)\le \delta (\G)$, in this case.

On the other hand, if $v=w$ then any directed walk from $(v,i)$ to $(v,j)$ induces a closed walk in $\G$ 
which begins and ends in $v$. Conversely, any closed walk in $\G$ which begins and ends in $v$ induces a directed walk from $(v,i)$ to $(v,j)$. 
Hence, we obtain that $d_{\G \otimes \mathring{K}_m}((v,i),(v,j))$ is equal to the length of the minimal directed cycle in $\G$ which begins and ends in $v$. Moreover, since $\G$ is arc-transitive we have that
$d_{\G \otimes \mathring{K}_m}((v,i),(v,j))=\gamma(\G)$ . 
This implies that 
$$\delta (\G \otimes \mathring{K}_m) = d_{\G \otimes \mathring{K}_m}((v,i),(v,j)) = \gamma(\G).$$
Notice that in this case, necessarily $\gamma(\G)> \delta(\G)$ since otherwise the walk which realizes the diameter in $\G$, would induce a minimal directed walk from two vertices in $\G\otimes \mathring{K}_m$ of length greater than the diameter of $\G \otimes \mathring{K}_m$.
Then, we obtain that
	$$\delta (\G\otimes \mathring{K}_m)=
\begin{cases}
	\hfil \delta(\G), & \qquad \text{if $\gamma(\G)\le \delta(\G)$ or $m=1$},	\\[1mm]
	\hfil \gamma(\G), & \qquad \text{if $\gamma(\G)> \delta(\G)$ and $m\ge 2$}.
\end{cases}$$
 Finally, by Lemma \ref{g=d+1} we get \eqref{diam2} as we wanted.
\end{proof}

We finish the section with an observation about the girth of the graphs of our interest.
\begin{lem} \label{lem 3gRWp}
Let $R$ be a commutative ring with identity of characteristic $char(R)=p^s$ with $p$ prime and  $s\in \N$. 
Then, we have that 
	\begin{equation} \label{gWG}
		3\le \gamma(W_R(k)) \le \gamma(G_R(k)) \le p^s.
	\end{equation}
	In particular, if $R=\ff_q$ is a field, then $3 \le \gamma(\G(k,q)) \le p$. 
\end{lem}

\begin{proof}
For every $a\in R$, the vertices $a_0=a$ and $a_1=a+1_R$ in $R$ form a directed edge from $a$ to $a_1$ in $G_R(k)$ since 
$a_1-a = 1_R = (1_R)^k \in U_R(k)$ for any $k$. If we recursively define  $a_{i+1}=a_i+1_R$, we have that $w_a(m)=aa_1a_2 \cdots a_m$ is a directed walk of length $m$ in $G_R(k)$. Since the characteristic of $R$ is $p^s$, we have that 
	$$a_{p^s} = a + p^s \cdot 1_R = a$$ 
and thus $w_a(p^s)$ is a $p^s$-cycle. In this way we get $3 \le \gamma(G_R(k)) \le p^s$. Since $G_R(k)$ is a subgraph of $W_R(k)$ we obtain \eqref{gWG}, as we wanted to show.  
The remaining assertion is clear since $char(\ff_q)=p$ and $G_{\ff_q}(k)=\G(k,q)$.
\end{proof}

\section{The Waring number $g_R(k)$ for $R$ local} \label{sec4}
Here we show that the Waring problem over finite commutative local rings can be reduced to the Waring problem over finite fields. 
More precisely, the Waring number $g_R(k)$ over the finite commutative local ring $(R,\frak m)$, if exists, is given in terms of the Waring number $g(k,q)$ over the residue field $R/\frak m\simeq \ff_q$.

We first need a lemma comparing the distances between two vertices $a,b \in R$ in the graphs $W_R(k)$ and $G_R(k)$. 
\begin{lem} \label{dW=dG}
Let $(R,\frak m)$ be a finite commutative local ring with residue field $\ff_{q}$ of prime characteristic $p$. 
Let $k \in \N$  be such that $(k,p)=1$ and suppose that $\frac{q-1}{(k,q-1)}\dagger q-1$.
Then, for any $a,b \in R$ we have 
$$d_{W_R(k)}(a,b) \leq d_{G_R(k)}(a,b)$$ 
and, furthermore,  $d_{W_R(k)}(a,b)=  d_{G_R(k)}(a,b)$ if $b \not \in a+\frak{m}$.  
\end{lem}

\begin{proof}
The fact that $\frac{q-1}{(k,q-1)}$ is a primitive divisor of $q-1$ implies that  $G_{R}(k)$ is connected and so $W_{R}(k)$ is connected by Corollary \ref{-1RFq} (here we use the hypothesis $(k,p)=1$). Hence, both distances $d_{W_R(k)}(a,b)$ and 
$d_{G_R(k)}(a,b)$ exist for any $a,b \in R$.
	
Since $G_{R}(k)$ is a subgraph of $W_{R}(k)$, we have that $d_{W_R(k)}(a,b) \le d_{G_R(k)}(a,b)$ for any $a,b \in R$. 
We now prove the equality in the case $a$ and $b$ are in different cosets. Thus,
assume that $b \not \in a+\frak m$. It is enough to show that
$d_{W_R(k)}(c,d) \ge d_{G_R(k)}(c,d)$.
	
Let $w_s =a_{0} \cdots a_{s}$ be a directed walk in $W_{R}(k)$ from $c=a_0$ to $d=a_s$ realizing the distance. 
Hence, we have that $a_{i+1}-a_i \in S_R(k)$ for every $i=0,\ldots,s-1$ and 
	$$d_{W_R(k)}(c,d)=s.$$
Clearly, if $s=1$, then $d_{W_R(k)}(c,d)=1=d_{G_R(k)}(c,d)$. So, we can assume that $s\ge 2$.
Notice that if $a_{i+1}-a_i \not\in \frak{m}$ for all $i$, then $w_s$ is just a walk in $G_{R}(k)$, 
since in that case $a_{i+1}-a_i \in U_R(k)$ for all  $i$. 
Hence, we may assume that there is some index $\ell$ such that $a_{\ell+1}-a_\ell \in \frak{m}$. Since $c-d\not\in \frak{m}$ and $s\ge 2$, there exists some index $i\in \{1,\ldots,s-1\}$ such that
	$$a_{i}-a_{i-1} \in \frak{m} \quad \text{and} \quad a_{i+1}-a_{i}\not\in\frak{m} \qquad \text{or else} \qquad   a_{i}-a_{i-1}\not\in\frak{m} \quad \text{and} \quad a_{i+1}-a_{i}\in\frak{m}.$$
Without loss of generality, suppose that $a_{i}-a_{i-1}\in\frak{m}$ and $a_{i+1}-a_{i}\not\in\frak{m}$.
By the choice of the $a_j$'s, we have that $a_{i+1}-a_{i}\in U_R(k)$ in this case. Then, by ($b$) in Lemma \ref{prop local} we have that $a_{i+1}-a_{i-1}\in U_R(k)$ since	$a_i\equiv a_{i-1}\pmod{\frak{m}}$. 
Therefore, by deleting the arrows $a_{i-1}a_{i}$, $a_{i}a_{i+1}$ and adding $a_{i-1}a_{i+1}$ we obtain a directed walk in $W_{R}(k)$ from $c$ to $d$ of length $s-1$, which contradicts the minimality of $s$. 
Thus, $a_{i+1}-a_i \not\in \frak{m}$ for all $i=1,\ldots,s-1$ and then the walk $w_s=a_{0} \cdots a_{s}$ is a walk in $G_{R}(k)$. By minimality of the distance we have that $d_{G_{R}(k)}(c,d)\le s = d_{W_{R}(k)}(c,d)$. Therefore, $d_{W_R(k)}(a,b)=  d_{G_R(k)}(a,b)$ as asserted.  
\end{proof}

We now relate the girth of $\G(k,q)$ with the distance between vertices of $G_R(k)$ and $W_{R}(k)$ which are in the same coset.
\begin{lem} \label{dW=dG dir}
Let $(R,\frak m)$ be a finite commutative local ring with residue field $\ff_{q}$ and put $m=|\mathfrak{m}|$.
Let $k \in \N$  be such that $(k,p)=1$ and suppose that $\frac{q-1}{(k,q-1)}\dagger q-1$.
If $m>1$, then we have the following cases:
 	\begin{enumerate}[$(a)$]
 		\item If $-1 \in U_R(k)$, then $d_{G_R(k)}(a,b)=2$ for any $a,b \in R$ such that $a\equiv b\pmod{\frak{m}}$. \sk 
 		
 		\item If $-1 \not \in U_R(k)$, then 	
 		$d_{G_R(k)}(a,b) = \gamma(\G(k,q))$ for any $a,b \in R$ such that $a\equiv b\pmod{\frak{m}}$.
 		Moreover, there exist $a,b \in R$ with $a\equiv b\pmod{\frak{m}}$ such that 
 		$d_{W_R(k)}(a,b) = \gamma(\G(k,q))$. 
  	\end{enumerate}
\end{lem}
 
\begin{proof}
($a$) This is a consequence of Proposition~\ref{Thm mod Local case} and Theorem \ref{Local case} (here we need $(k,p)=1$). 
 	In fact, if $a\equiv b \pmod{\frak m}$ then they are not neighbors in $G_{R}(k)$. 
 	Also, if $c\in R$ is such that $\bar{a}$ and  $\bar{c}$  (the class of $a$ and $c$ in the residual field $\ff_{q}$) are neighbors
 	in $\G(k,q)$ then $c$ is a common neighbor of $a$ and $b$ by Proposition~\ref{Thm mod Local case} since, in this case, $\G(k,q)$ and $G_{R}(k)$ are both undirected graphs because of the assumption $-1 \in U_R(k)$. Therefore, we have that
 	$d_{G_R(k)}(a,b)=2$, as asserted.
 	
\noindent ($b$) 
Assume now that $-1\not \in U_R(k)$, hence the graphs $G_R(k)$ and $W_R(k)$ are directed. Notice that in this case we have that
 	$$U_R(k) \cap (-U_R(k)) = \varnothing \qquad \text{and} \qquad S_R(k) \cap (-S_R(k)) = \varnothing,$$ 
so that we cannot replicate the proof for the undirected case, 
since if there is an arc from $\bar{a}$ to $\bar{c}$ in $\G(k,q)$, then there is not an arc from $\bar{c}$ to $\bar{a}$, for any $c\in R$. 
 	
Observe that any directed walk from $a$ to $b$ in $G_{R}(k)$ induces a directed walk from $\bar{a}$ to $\bar{b}$ in $\G(k,q)$ of the same length, but since $\bar{a}=\bar{b}$, this is a directed cycle which begins and ends in $\bar{a}$. Reciprocally, any directed cycle which begins and ends in $\bar{a}$ induces a directed walk from $a$ to $b$ with the same length (not necessarily unique). So, 
$d_{G_{R}(k)}(a,b)$ is exactly the length of the minimum directed cycle which begins and ends in $\bar{a}$. Now by taking into account that $\G(k,q)$ is an arc-transitive graph, we have that the girth can be realized in any vertex of the graph, that is $d_{G_{R}(k)}(a,b)=\gamma(\G(k,q))$, as asserted. 
	
Finally, by Lemma \ref{dW=dG} we have that  
	$$d_{W_{R}(k)}(a,b)\le d_{G_{R}(k)}(a,b)=\gamma(\G(k,q)).$$
Notice that in $W_{R}(k)$ may exist arcs in any coset of $\frak{m}$ in $R$ and so two vertices in the same coset of $\frak{m}$ could be neighbors in $W_{R}(k)$.
Without loss of generality we can assume that $c=0$ and so it is enough to find some $d\in \frak{m}$ such that $d_{W_{R}(k)}(0,d)=\gamma(\G(k,q))$.
	
For this purpose, notice that the assumption $-1\not \in U_R(k)$ implies that $k>1$, since $-1 \in U_{R}(k)$ otherwise. On the other hand, if we consider the $s$-th power of $\frak{m}$
	$$\frak{m}^{s}= \Big\{  \sum_{1\le i \le t} n_{i,1} n_{i,2} \cdots n_{i,s}  : n_{i,1}, n_{i,2}, \ldots, n_{i,s} \in \frak{m}, t \in \mathbb{N} \Big\},$$
since $\frak{m}$ has only nilpotents elements, then there exists $M \in \mathbb{N}$ such that $\frak{m}^{M}=\{0\}$. This implies that $\frak{m}^2 \subsetneq \frak{m}$, since if $\frak{m}^2=\frak{m}$ then $\frak{m}^{\ell}=\frak{m}$ for all $\ell \in \mathbb{N}$.
	
\noindent \textit{Claim}: If $d\in \frak{m}\smallsetminus \frak{m}^2$, then $d$ cannot be written as a sum of $k$-th powers of nilpotent elements.

Since $k>1$, we have that if $x\in \frak{m}$ then $x^k\in \frak{m}^2$, then if we assume that
	$$d=x_1^{k}+ \cdots + x_{s}^k, \qquad x_1,\ldots,x_s \in \frak{m},$$
then $d\in \frak{m}^2$ since $\frak{m}^2$ is an ideal, which is a contradiction. 
Therefore, $d$ cannot be written as a sum of $k$-th powers of nilpotent elements.

So, let $d\in \frak{m}\smallsetminus \frak{m}^2$.
The above claim shows that there is not a directed walk in $W_{R}(k)$ from $0$ to $d$
induced only for nilpotents arrows. That is if $w_s=a_0\cdots a_{s}$ is a walk from $0$ to $d$ in $W_{R}(k)$, then $a_{i+1}-a_{i}\not \in \frak{m}$ for some index $i$.
Now if we take a walk as above $w_s$ from $0$ to $d$, realizing the distance $d_{W_R(k)}(0,d)$. 
In the same way as in Lemma \ref{dW=dG}, we can show that
$a_{i+1}-a_i\not \in \frak{m}$ for $i=0,\ldots,s-1$.
Hence, the walk $w_s$ is a walk in $G_{R}(k)$ from $0$ to $d$ and hence
	$$\gamma(\G(k,q))=d_{G_{R}(k)}(0,d)\le s =d_{W_{R}(k)}(0,d).$$ 
Therefore, we obtain that $d_{W_{R}(k)}(0,d)=\gamma(\G(k,q))$, 
as asserted.
 \end{proof}

Recall that the graph $G_R(k)$ is undirected if and only if $-1\in U_R(k)$ and the conditions for that are given in $(a)$ of Corollary \ref{-1RFq}. Thus, we have 
\begin{equation} \label{nodirigido}
G_R(k) \text{ is undirected } \quad \Leftrightarrow \quad \text{if $q$ is even or else if $q$ is odd and $(k,q-1) \mid \tfrac{q-1}2$}.
\end{equation} 

We can now state and prove the announced result.
\begin{thm} \label{teo waring Local case}
Let $(R,\frak m)$ be a finite commutative local ring with identity with residue field $R/\frak m \simeq \ff_{q}$ of prime characteristic $p$. 
Let $k \in \N$ be such that $(k,p)=1$ and put $k'=(k,q-1)$. 
If $\frac{q-1}{k'} \dagger q-1$, then $g_R(k)$ exists, $W_R(k)$ is connected, and 
	$$\delta(W_{R}(k))=g_R(k)= \delta(G_{R}(k)).$$ 
Moreover, in this case we have the following:
	\begin{enumerate}[$(a)$]
		\item If $G_R(k)$ is undirected, then 
		\begin{equation} \label{gR local}
			g_{R}(k)= 
			\begin{cases}
				\hfil 1, & \qquad \text{if $k'=1$ and $|\mathfrak{m}| = 1$},	\\[1mm]
				\hfil 2, & \qquad \text{if $k'=1$ and $|\mathfrak{m}| \ge 2$}, \\[1mm] 
				g(k,q), 
				& \qquad \text{if $k'>1$}. 
			\end{cases}
		\end{equation}
		
		\item If $G_R(k)$ is directed, then 
		\begin{equation} \label{gR local dir}
		g_{R}(k) = 
				\begin{cases}
					\hfil g(k,q), & \qquad \text{if $\gamma(\G(k,q)) \le g(k,q)$ or $|\mathfrak{m}|=1$},	\\[1mm]
					\hfil g(k,q)+1, & \qquad \text{if $\gamma(\G(k,q))> g(k,q)$ and $|\mathfrak{m}| \ge 2$},  
				\end{cases} 
		\end{equation}
where $\G(k,q)$ is the GP-graph defined in \eqref{GPkq}.
\end{enumerate} 
 Therefore, $g_R(k)=g_{R'}(k)$ for any pair of finite commutative local rings $R,R'$ having the same (up to isomorphism) residue field. 
\end{thm}

\begin{proof}
The hypothesis $\frac{q-1}{k'} \dagger q-1$ implies that $g(k,q)$ exists, and so the graphs $G_{\ff_q}(k)$ and $G_{R}(k)$ are connected by Theorem~\ref{Local case} and Proposition~\ref{lem1} (here we need $(k,p)=1$). 
Therefore $W_{R}(k)$ is also connected by ($b$) in Theorem \ref{eqbound}.
We recall that it is known that $g(k,q)=g(k',q)$, that $\G(1,q)$ is the complete graph $K_q$ and that $g(1,q)=1$.

Now, by Theorem \ref{eqbound} we have that $\delta(W_{R}(k))\leq \delta(G_{R}(k))$ and hence
it is enough to prove that $\delta(G_{R}(k))\leq \delta(W_{R}(k))$.	

If $m=1$, then $R$ is a finite field and so $W_{R}(k)=G_R(k)$ and certainly 
$\delta(W_{R}(k)) = \delta(G_R(k))$.
Thus, we assume that $m>1$. 
The hypothesis $k>1$ and $m>1$ imply that $W_{R}(k)$ is not the complete graph and so 
$$\delta(W_{R}(k))\ge 2.$$

Now, let $a,b\in R$ be two vertices in $G_R(k)$ realizing the diameter, that is we have that 
$$\delta(G_{R}(k))=d_{G_R(k)}(a,b).$$ 
First, suppose that $b \in a+\frak{m}$. The item ($a$) of Lemma \ref{dW=dG dir} implies that
$d_{G_R(k)}(a,b)=2$ if $-1\in U_{R}(k)$, in this case we obtain that 
$$\delta(W_{R}(k))=2=\delta(G_{R}(k)).$$
Now suppose that $-1\not \in U_R(k)$, the item ($b$) of Lemma \ref{dW=dG dir} implies that 
$$\delta(G_{R}(k))=d_{G_{R}(k)}(a,b)=\gamma(\G(k,q)),$$ 
the same item says that 
there exist $c,d \in R$ with $d\in c+\frak{m}$ such that $d_{W_R(k)}(c,d)=\gamma(\G(k,q))$, by maximality of the diameter we obtain that
$$\delta(G_{R}(k))=\gamma(\G(k,q))=d_{W_R(k)}(c,d)\le \delta(W_{R}(k)).$$
Therefore $\delta(G_{R}(k))=\delta(W_{R}(k))$, in this case as well.

We can now assume that $a,b$ are not in the same coset of $\frak{m}$ in $R$,  thus 
by Lemma \ref{dW=dG} we have that 
$$\delta(G_{R}(k))=d_{G_R(k)}(a,b)= d_{W_R(k)}(a,b) \le \delta(W_{R}(k)).$$
Therefore, we obtain that $\delta(G_{R}(k))=\delta(W_{R}(k))$, as asserted.

To prove \eqref{gR local}, since $(k,p)=1$ we have the decomposition $G_R(k)=G_{\ff_q}(k) \otimes \mathring{K}_m$ and hence we can apply Lemma~\ref{lem1}, thus obtaining 
	$$\delta(W_R(k))=\delta(G_R(k))=\delta(G_{\ff_q}(k))=g(k,q).$$
	
Finally, the assertion \eqref{gR local dir} is a direct consequence of item ($c$) of Lemmas~\ref{lem1}, \ref{dW=dG} and \ref{dW=dG dir}. 
The remaining assertions in the statement are obvious from \eqref{gR local}, and the result follows.
\end{proof}

\begin{rem}
Theorem \ref{teo waring Local case} reduces the study of Waring numbers over local rings to Waring numbers over finite fields. The advantage of this is that pretty much is known for the numbers $g(k,q)$. A recent list of the main known exact values, upper and lower bounds for $g(k,q)$ is given in Section 2 of \cite{PV6} (see also the handbook \cite{MP}).
\end{rem}

Notice that by the previous theorem we know which are all possible values that the Waring number $g_R(k)$ can take for any local ring of fixed size. In fact, if $(R,\frak m)$ is a local ring with residue field $R/\frak m \simeq \ff_q$ and 
$|R|=p^t$ for some prime number $p$, then 
$$g_R(k) \in V \cup (V+1) \qquad \text{ where } \qquad V = \{1,g(k,p), g(k,p^2), \ldots, g(k,p^{t-1})\}.$$

We can adapt any bound for Waring numbers over finite fields to Waring numbers over rings.
\begin{rem}
Assume that $R$ is a finite commutative local ring with identity with residue field isomorphic to $\ff_{q}$ of prime characteristic $p$ and that $g_R(k)$ exists. 

\noindent ($i$) 
By Theorem \ref{teo waring Local case}, we have that 
	$$g_R(k) \le g(k,q)+1.$$ 
We can combine this with any bound for $g(k,q)$. For instance, if $a(k,q) \le g(k,q) \le b(k,q)$ then we have
\begin{equation} \label{bounds}
	a(k,q) \le g_R(k) \le b(k,q)+1.
\end{equation}
  	
\noindent ($ii$) 
Using a result in \cite{GlR}, Cipra proved that if $g(k,q)$ exists and $k < \sqrt q$ then 
$g(k,q) \le 8$ (see \cite[Corollary 1]{Ci}). Cipra also found in \cite[Theorem 4]{Ci} the following bounds 
$g(k,p^2) \le 16 \sqrt{k+1}$ and $g(k,p^m) \le 10 \sqrt{k+1}$ for $n\ge 3$. By Cipra's bounds and  
\eqref{bounds} we arrive at 
\begin{equation} \label{bounds2}
	g_R(k) \le \begin{cases}
		\hfil	9, 		& \qquad \text{if $1 < k < \sqrt{q}$}, \\[1mm]
		a\sqrt{k+1}+1,	& \qquad \text{if $\sqrt{q} \le k \le q-1$}, \end{cases}
\end{equation}
where $a=16$ if $q=p^2$ and $a=10$ if $q=p^m$ with $m\ge 3$.
\end{rem}

\subsubsection*{The Waring's function on rings}
Let $R$ be any ring and $k$ any positive integer and consider the set of pairs $(R,k)$ such that $g_R(k)$ exists, that is 
	$$\mathbb{W}_{rings} = \{(R,k):g_R(k) \in \N\} \subset \mathcal{R} \times \N,$$ 
where $\mathcal{R}$ is the set of rings.
The \textit{Waring's function on rings} is 
	$$g: \mathbb{W}_{rings} \longrightarrow \N, \qquad (R,k) \mapsto g_R(k).$$
In \cite{PV6}, we proved that $g$ is a surjective function if we restrict the domain to finite fields. 
We now prove that $g$, restricted to commutative rings with identity which are not fields, is also surjective. In fact, it will be enough to restrict ourselves to  (finite commutative) local rings.

\begin{prop} \label{g sobre}
	Given any $b\in \N$ there exist $k \in \N$ and a ring $R$ such that $g_R(k)=b$.
	Moreover, $R$ can be taken to be either a finite field or a finite commutative local ring with identity which is not a field.
\end{prop}

\begin{proof}
	Since $g_R(1)=1$, we can assume that $b\ge 2$. If $R=\ff_q$ is a finite field, then Proposition~4.7 in \cite{PV6} ensures that there is some $k \in \N$ such that $g_R(k)=g(k,q)=b$. So suppose that $R$ is not a finite field.
	Let $p$ be any prime which is coprime with $b$. Thus $p$ is a unit in $\Z_b$ and there is $a\in \N$ such that $p^a\equiv 1 \pmod b$.
	By Lemma 4.2 and Corollary 4.3 in \cite{PV6} we have that $b \mid \frac{p^{ab}-1}{p^a-1}$ and $g(k,q)=b$ where $k= \frac{p^{ab}-1}{b(p^a-1)}$ and $q=p^{ab}$. 
	Let $R$ be any commutative local ring with identity whose residue field is isomorphic to $\ff_q$. Now, since $k':=(k,q-1)=k>1$, by \eqref{nodirigido} we have that $G_R(k)$ is non-directed. 
	Taking $c=p^a-1$ in Lemma 3.3 of \cite{PV7}, we get that $\frac{q-1}{k'}=b(p^a-1)$ is a primitive divisor of $p^{ab}-1$.
	Thus, by the third case in ($a$) of Theorem \ref{teo waring Local case} we have that
	$g_R(k)=g(k,q)=b$, as we wanted to see.
\end{proof}

\section{A reduction formula for $g_R(k)$ with $R$ a local ring} \label{sec5red}
In this short section we present a reduction formula for Waring numbers over a finite commutative local ring with identity $R$. In some particular cases, the Waring number over a finite commutative local ring $R_{ab}$ of size $p^{ab}$ can be expressed as $b$ times the Waring number over a finite local commutative local ring of size $p^a$, that is 
$$g_{R_{ab}}(k) = b \cdot g_{R_a}(k),$$ 
for suitable values of $k$. 
This kind of result holds for finite fields in the case when the associated GP-graph $\G(k,p^{ab})$ is Cartesian decomposable 
(see \cite{PV7}) and we will extend it to the case of finite commutative local rings with identity.

Cartesian decomposable GP-graphs, that is those $\G(k,q)$ that can be decomposed as a Cartesian product for certain graphs 
$\G_1, \ldots, \G_t$, i.e.\@
$$\G(k,q) = \G_1 \square \cdots \square \G_t$$ 
were studied by Pearce and Praeger. In \cite{PP} they showed that $\G(k,p^m)$ is Cartesian decomposable if and only if 
	$$\tfrac{p^m-1}{k}=bc \qquad \text{with} \qquad b>1, \: b\mid m \: \text{ and } \: c \dagger p^{\frac mb}-1.$$ 
	In this case, $\G$ turns up to be the product of a single GP-graph, that is 
	$$\G \simeq \square^b \G_0 \qquad \text{with} \qquad \G_0 = \G \big(\tfrac{p^{\frac mb}-1}{c},p^{\frac mb}\big),$$ 
where $b$ and $c$ are integers as described above.

We begin with the following basic result on the girth of iterated Cartesian products of graphs.
\begin{lem} \label{girth prod}
 Let $\G=\square^b \G_0$ be the Cartesian product of $b$ copies of a fixed graph $\G_0$ with cycles. 
 Then, the girths of $\G$ and $\G_0$ satisfy 
 \begin{equation} \label{des girths}
 	\gamma(\G) \le \gamma(\G_0).
 \end{equation}
Moreover, if $\G_0$ is undirected, then $\gamma(\G)=3$ if $\gamma(\G_0)=3$ or $\gamma(\G)=4$ if $\gamma(\G_0)\ge 4$. 
\end{lem}

\begin{proof}
We consider the cases when $\G$ is directed or undirected separately. 
If $\G$ is undirected, then the girth of $\G$ is the minimum between the girth of $\G_0$ and the length of the minimum $n$-cycle $C_n$ generated by the product. Notice that there are always $4$-cycles in $\G$. In fact, if $v_1w_1$ and $v_2w_2$ are two edges in $\G_0$ then the vertices ${\bf v}_1=(v_1,v_2,v_3,\ldots,v_b)$, ${\bf v}_2=(w_1,v_2,v_3,\ldots,v_b)$, 
${\bf v}_3=(w_1,w_2,v_3,\ldots,v_b)$ and ${\bf v}_4=(v_1,w_2,v_3,\ldots,v_b)$ form the $4$-cycle ${\bf v}_1{\bf v}_2{\bf v}_3{\bf v}_4{\bf v}_1$ in $\G$. In this way, we have that
$$\gamma(\G) = \min\{\gamma(\G_0), \gamma(C_4) \} = \min\{\gamma(\G_0), 4 \},$$  
from which \eqref{des girths} follows.

On the other hand, if $\G_0$ is directed, then $\G$ is also directed.
Since $\G_0$ has cycles by hypothesis, then  $\gamma(\G_0)$ exists, so we can consider 
$v_0 v_1 \cdots v_{\ell}$  with $v_{0}=v_{\ell}$ a directed cycle in $\G_0$ 
of length $\ell= \gamma(\G_0)$. Now, let $v$ be a fixed vertex in $\G_0$ and let us consider the vertices 
$${\bf v}_i = (v_i,v,\ldots v)\in V(\G) \qquad \text{for $i=0,\ldots,\ell$}.$$
Hence, we have that ${\bf v}_0={\bf v}_\ell$ and ${\bf v}_0 {\bf v}_1 \cdots {\bf v}_{\ell-1} {\bf v}_{\ell}$ is a directed closed walk in $\G$ of length $\gamma(\G_0)$. Thus, by minimality of the girth we obtain that
$\gamma(G)\le \gamma(\G_0)$, as asserted.
\end{proof}

We are now in a position to prove the reduction formula for Waring numbers over local rings.

\begin{thm} \label{teo reduction}
Let $p$ be a prime and $a,b,c \in \N$ with $b>1$ such that $c \dagger p^a-1$ and $bc \dagger p^{ab}-1$. 
Let $R_{ab}$ and $R_{a}$ be any finite commutative local rings with identity whose residue fields have sizes $p^{ab}$ and $p^a$, respectively. 
Thus, both Waring numbers $g_{R_{ab}}(\tfrac{p^{ab}-1}{bc})$ and $g_{R_a}(\tfrac{p^{a}-1}{c})$ exist and the following reduction formulas hold for them.
	\begin{enumerate}[$(a)$]
		\item If either $p$ is even or else $p$ is odd and $bc$ is even, then we have
			\begin{equation} \label{red fla}
				g_{R_{ab}}(\tfrac{p^{ab}-1}{bc}) = b g_{R_a}(\tfrac{p^a-1}c).
			\end{equation} 
		\item If $p$ is odd and $bc$ is odd  then we have
			\begin{equation} \label{red fla dir}
			g_{R_{ab}}(\tfrac{p^{ab}-1}{bc}) = 	\begin{cases}
			\hfil bg_{R_a}(\tfrac{p^a-1}c), 	& \qquad \text{if $\gamma(\G(\tfrac{p^a-1}c,p^a)) \le g(\tfrac{p^a-1}c,p^a)$ or $m_a=1$},	\\[1.75mm]
			  b(g_{R_a}(\tfrac{p^a-1}c)-1), 	& \qquad \text{if $\gamma(\G(\tfrac{p^a-1}c,p^a)) > g(\tfrac{p^a-1}c,p^a)$ and $m_a\ge 2$},  
			\end{cases} 
			\end{equation}  
where $m_a$ is the size of the maximal ideal of $R_{a}$. 
	\end{enumerate}	
\end{thm}

\begin{proof}
Put $k_{ab}=\tfrac{p^{ab}-1}{bc}$ and $k_a=\tfrac{p^a-1}c$, for simplicity, and notice that $(k_a,p)=1$ and $(k_{ab},p)=1$.	
Since $k_{a}'=(k_{a},p^{ab}-1)=k_{a}$ and $k_{ab}'=(k_{ab},p^{ab}-1)=k_{ab}$, the hypotheses $c \dagger p^a-1$ and $bc \dagger p^{ab}-1$ imply that 
	$$\tfrac{p^a-1}{k_a'} \dagger p^a-1 \qquad \text{and} \qquad \tfrac{p^{ab}-1}{k_{ab}'} \dagger p^{ab}-1,$$ 
so we are in the conditions of Theorem \ref{teo waring Local case} and in particular $g_{R_{ab}}(\tfrac{p^{ab}-1}{bc})$ and $g_{R_a}(\tfrac{p^{a}-1}{c})$ exist. 
Moreover, we have the Cartesian decomposition 
	$$\G(\tfrac{p^{ab}-1}{bc}, p^{ab}) \simeq \square^b \G(\tfrac{p^{a}-1}{c},p^a).$$
We now distinguish between the directed and the undirected case for these graphs.

\sk 

\noindent ($a$) 
Since either $p$ is even or else $p$ is odd and $bc$ is even, all the graphs $G_{R_a}(k_a)$, $G_{R_{ab}}(k_{ab})$, $\G(k_a,p^a)$ and $\G(k_{ab},p^{ab})$ are undirected. By ($a$) in Theorem \ref{teo waring Local case} and Proposition 2.1 in \cite{PV7} we have that
	\begin{equation} \label{fla red for R}
		g_{R_{ab}}(\tfrac{p^{ab}-1}{bc}) = g(\tfrac{p^{ab}-1}{bc}, p^{ab}) = b\, g(\tfrac{p^{ab}-1}{bc}, p^{ab}) = b\, g_{R_{a}}(\tfrac{p^{a}-1}{c}),
	\end{equation}
and hence \eqref{red fla} holds. 

\noindent ($b$) 
Assume now that both $p$ and $b$ are odd. Thus, all the graphs $G_{R_a}(k_a)$, $G_{R_{ab}}(k_{ab})$, $\G(\tfrac{p^a-1}c,p^a)$ and $\G(\tfrac{p^{ab}-1}{bc},p^{ab})$ are directed, and hence we are in the situation of ($b$) in Theorem \ref{teo waring Local case}.

If $m_a=1$, we are in the finite field case, and hence \eqref{fla red for R} holds by Theorem 2.4 in \cite{PV7}.
So we can assume that $m_a\ge 2$.

If $\gamma(\G(\tfrac{p^a-1}c,p^a))\le g(\tfrac{p^a-1}c,p^a)$ then, since $b>1$, we have that 
$$\gamma(\G(\tfrac{p^{ab}-1}{bc},p^{ab}))\le g(\tfrac{p^a-1}c,p^a)< b g(\tfrac{p^a-1}c,p^a)=g(\tfrac{p^{ab}-1}{bc},p^{ab}).$$
Hence, by ($b$) in Theorem \ref{teo waring Local case} and Theorem 2.4 from \cite{PV7}, 
we obtain that
$$g_{R_{ab}}(\tfrac{p^{ab}-1}{bc})=g(\tfrac{p^{ab}-1}{bc},p^{ab}) =bg(\tfrac{p^{a}-1}{c},p^a)=bg_{R_{a}}(\tfrac{p^{a}-1}{c}),$$
as asserted.

Now, if $\gamma(\G(\tfrac{p^a-1}c,p^a))> g(\tfrac{p^a-1}c,p^a)$, then we have that 
	\begin{equation}\label{eq gRa g+1}
		g_{R_{a}}(\tfrac{p^{a}-1}{c}) = g(\tfrac{p^{a}-1}{c},p^a)+1,
	\end{equation}
by \eqref{gR local dir}.
On the other hand, Lemmas \ref{g=d+1} and \ref{girth prod} imply that 
	$$\gamma(\G(\tfrac{p^{ab}-1}{bc},p^{ab})) \le  \gamma(\G(\tfrac{p^a-1}c,p^a))=g(\tfrac{p^{a}-1}{c},p^a)+1,$$
and thus, since $b>1$, we have 
	$$g(\tfrac{p^{a}-1}{c},p^a)+1\le bg(\tfrac{p^{a}-1}{c},p^a)=g(\tfrac{p^{ab}-1}{bc},p^{ab})$$ 
so that we obtain 
	$$\gamma(\G(\tfrac{p^{ab}-1}{bc},p^{ab}))\le g(\tfrac{p^{ab}-1}{bc},p^{ab}).$$
Finally, ($b$) in Theorem \ref{teo waring Local case} implies that $g_{R_{ab}}(\tfrac{p^{ab}-1}{bc})=g(\tfrac{p^{ab}-1}{bc},p^{ab})$. 
Therefore, by \eqref{eq gRa g+1} we get
	$$g_{R_{ab}}(\tfrac{p^{ab}-1}{bc})= bg(\tfrac{p^{a}-1}{c},p^a)= b(g_{R_{a}}(\tfrac{p^{a}-1}{c})-1)$$
in this case, and the result is proved.
\end{proof}

\section{Explicit formulas on $G_R(k)$ for $R$ local} \label{sec5}
In this and the next section we apply the results from the two previous ones to some particular cases, and we get explicit computations for Waring numbers over local rings. Here we adapt the already known results in the literature for finite fields to finite commutative local rings.

Since the numbers $g(k,q)$ corresponds to diameters of GP-graphs $\G(k,q)$, using Theorem \ref{teo waring Local case}, we can adapt any of the known results for Waring numbers over the finite field $\ff_q$ for Waring numbers over a local ring $(R,\frak m)$ having residue field $R/\frak m \simeq \ff_q$. 
For instance, this is the case for the Kononen's result (\cite{Kon}) and all the explicit values $g(k,q)$ obtained in \cite{PV6}.

\subsection{Kononen's result for local rings}
We first extend Kononen's result on Waring numbers over finite fields to any finite commutative local ring with identity. Recall that an integer $b$ is a \textit{primitive root modulo} $n$ if for every integer $a$ coprime to $n$, there is some integer $k$ such that  $b^k \equiv a \pmod n$. That is, $b$ is a primitive root modulo $n$ if and only if $\Z_n^*$, the multiplicative group of integers modulo $n$, is cyclic and $b$ is a generator of $\Z_n^*$.

\begin{prop} \label{kononen} 
Let $p$ and $r$ be primes such that $p$ is a primitive root modulo $r^t$ for some $t \in \N$. 
Let $R$ be any finite commutative local ring with identity whose residue field $\ff_q$ has size $q=p^{\varphi(r^t)}$, where $\varphi$ denotes the Euler's totient function. Then, we have  
\begin{equation} \label{kon fla}
	g_R \big( \tfrac{p^{\varphi(r^t)}-1}{r^t} \big) = \tfrac 12 (p-1) \varphi(r^t).
\end{equation} 
 If in addition $p$ and $r$ are odd, then
\begin{equation} \label{kon fla2}
	g_R \big( \tfrac{p^{\varphi(r^t)}-1}{2r^t} \big) =
 		\begin{cases}
 			r^{t-1}\lfloor \tfrac{pr}4 -\tfrac{p}{4r}\rfloor, & \qquad \text{ if }	r<p,		\\[1.5mm]
 			r^{t-1}\lfloor \tfrac{pr}4 -\tfrac{r}{4p}\rfloor, & \qquad \text{ if } r\ge p. 
 		\end{cases}
 \end{equation} 
\end{prop}

\begin{proof}
Under the hypotheses in the statement, in 2010 the authors of \cite{KK} and \cite{Win2} obtained the expressions 
\begin{equation} \label{kon fla3} 
	g \big( \tfrac{p^{\varphi(r^t)}-1}{r^t},p^{\varphi(r^t)} \big) = \tfrac 12 (p-1) \varphi(r^t),
\end{equation} 
and, if in addition $p$ and $r$ are both odd integers, the more explicit ones
\begin{equation} \label{kon fla4} 
g \big( \tfrac{p^{\varphi(r^t)}-1}{2r^t},p^{\varphi(r^t)} \big) =
\begin{cases}
	r^{t-1}\lfloor \tfrac{pr}4 -\tfrac{p}{4r}\rfloor, & \qquad \text{ if }	r<p,		\\[1.5mm]
	r^{t-1}\lfloor \tfrac{pr}4 -\tfrac{r}{4p}\rfloor, & \qquad \text{ if } r\ge p. 
\end{cases}
\end{equation} 
Winterhof and van de Woestijne (\cite{Win2}) proved \eqref{kon fla3} and \eqref{kon fla4} in the case $t=1$ and then Kononen (\cite{KK}) proved the general case using the known result for $t=1$. 

The number $k=\tfrac{p^{\varphi(r^t)}-1}{r^t}$ is coprime with $p$ and $k'=(k,q-1)=k>1$. 
Since $p$ is a primitive root modulo $r^t$, we clearly have that $\frac{q-1}{k'} \dagger q-1$ and, hence,  
we are under the hypotheses of Theorem~\ref{teo waring Local case}. 

Now, $G_R(k)$ is undirected if $p=2$ or if $p$ is odd and $k \mid \frac{q-1}2$, which happens if $r$ is even.
Thus, if $p=2$ or $p$ is odd and $r$ is even, by $(a)$ in Theorem \ref{teo waring Local case} we have that $g_R(k)=
g(k,q)$ (since $k'=k>1$) and \eqref{kon fla} holds in this case. 

On the other hand, if $p$ and $r$ are odd, the graph $G_R(k)$ is directed. 
In this case, since the girth of the graph $\G(k,q)$ is less than or equal to the characteristic of the residue field, we have $\gamma(\G(k,q)) < p$. 
By \eqref{kon fla3} we have that $g(k,q) \ge p$ if and only if 
		$$ r^{t-1}(r-1) \ge \tfrac{2p}{p-1}.$$
This inequality holds for every odd $r$ and every $t$, except for the case $t=1$ and $r=3$.

Now, if $t=1$ and $r=3$, then $k=\frac{p^2-1}3$ and hence $U_{R/\frak m, k} = \{1,\beta,\beta^{2}\}$ where $\frak m$ is the maximal ideal of $R$ and $\beta=\alpha^k$ with $\alpha$ a primitive element of $\ff_{p^2}$ (since $q=p^{\varphi(3)}$). Notice that 
		$$\beta^3-1 = (1+\beta+\beta^{2})(\beta-1)=0,$$
and thus, since $\beta\neq 1$, we have that $1+\beta+\beta^{2}=0$. This implies that there exists a directed $3$-cycle, which begins and ends in $0$. This implies that 
$\gamma(\G(\frac{p^2-1}{3},p^2))\le 3$. Since $p$ is odd and is a primitive root modulo $3$,
we have that $p\equiv 2\pmod{3}$ and $p\ge 5$, and hence by \eqref{kon fla3} we get 
		$$g(\tfrac{p^2-1}{3},p^2)=p-1\ge 4.$$ 
Therefore $\gamma(\G(\frac{p^2-1}{3},p^2))\le g(\frac{p^2-1}{3},p^2)$ in this case, as well. 

In this way, since $\gamma(\G(k,q)) \le g(k,q)$, by ($b$) in Theorem \ref{teo waring Local case} we get that $g_R(k)=g(k,q)$, and hence \eqref{kon fla} is proved. 

For the remaining expression, we proceed similarly as before with for $\tilde k=\tfrac 12 k$. Since $p$ and $r$ are both odd, the graph $G_R(\tilde k)$ is undirected and hence \eqref{kon fla4} and ($a$) of Theorem \ref{teo waring Local case} together imply \eqref{kon fla2}, as we wanted to show.
\end{proof}

\subsection{Explicit results from the reduction formula}
We now adapt the two main results of \cite{PV6} for Waring numbers over finite fields, namely Theorems 4.1 and 6.1, to Waring numbers over finite commutative local rings, from which several particular cases can be deduced.

If $a$ and $b$ are integers, we denote by $ord_b(a)$ the order of $a$ modulo $b$; that is, the least integer $t$ such that $a^t \equiv 1 \pmod b$. 

\begin{thm} \label{prop gral fla b}
Let $p$ be a prime and let $a,b \in \N$ such that $b \mid \frac{p^{ab}-1}{p^a-1}$. 
For any finite commutative local ring with identity $R$ whose residue field $\ff_q$ has size $q=p^{ab}$ we have that 
\begin{equation} \label{gral fla b}
	g_R(\tfrac{p^{ab}-1}{b(p^a-1)}) = g_R(\tfrac 1b(p^{a(b-1)}+\cdots+p^{2a}+p^a+1))= b.
\end{equation}

Furthermore, 
\eqref{gral fla b} holds in the following cases:
  \begin{enumerate}[$(a)$]
	\item If $b=r$ is a prime different from $p$ and $p^a \equiv 1 \pmod r$. \sk 
	
	\item If $b=2r$ with $r$ an odd prime, $p^a$ coprime with $b$ and $p^a \equiv \pm1 \pmod r$. \sk 
	
	\item If $b=r r'$ with $r<r'$ odd primes such that $r \nmid r'-1$ and $p^a \equiv 1 \pmod{rr'}$. \sk 
	
	\item If $b=r_1 r_2 \cdots r_\ell$ with $r_1 < r_2 < \cdots < r_\ell$ primes different from $p$ with $p^a \equiv 1 \pmod{r_1}$ and $(p^a)^{b/r_i} \equiv 1 \pmod{r_i}$ for $i=2, \ldots,\ell$. \sk 
	
	\item If $b=r^t$ with $r$ prime such that $ord_{b}(p^a)=r^h$ for some $0\le h<t$. \sk 
	
	\item If $b = r_1^{t_1}	\cdots r_\ell^{t_\ell}$ with $r_1 < \cdots < r_\ell$ primes different from $p$ where $ord_{r_{i}^{t_i}}(p^a) = r_{i}^{h_i}$ with  $0\le h_i\le t_{i}-1$ for all $i$.
  \end{enumerate}
Conversely, if \eqref{gral fla b} holds with $b$ as in one of the items $(a)$--$(e)$ then the condition for $p^a$ stated in the 
corresponding item holds.
\end{thm}

\begin{proof}
In Theorem 4.1 in \cite{PV6} we showed that if $p$ is a prime and $a,b$ are positive integers 
such that $b \mid \frac{p^{ab}-1}{p^a-1}$, then $g(\tfrac{p^{ab}-1}{b(p^a-1)}, p^{ab}) = b$. If we take $k=\tfrac{p^{ab}-1}{b(p^a-1)}$, by using ($a$) in Theorem \ref{teo waring Local case} with $k'>1$ we get \eqref{gral fla b}. 
The remaining assertions follow directly from \eqref{gral fla b} and Theorem 6.1 in \cite{PV6}.
\end{proof}

\begin{rem}
($i$) By applying the same argument used in the above two proofs, all the results of the form 
$g(k,q)=b$ obtained in Examples 4.4 and 4.5, Corollaries 6.2--6.5, 6.8--6.10, 6.13, Proposition~6.11 and Examples~6.6, 6.7 and 6.14 in \cite{PV6} translate \textsl{mutatis mutandis} to $g_R(k)=b$ for any finite commutative local ring $(R,\frak m)$ whose residue field $R/\frak m$ has size the $q$. This gives a lot of explicit exact results for Waring numbers over finite rings.

\noindent ($ii$)
For instance, let $p$ be a prime and $a,b \in \N$ such that $p\nmid b$ and assume that $R$ is any finite commutative local ring with identity whose residue field is of size $q=p^{ab}$. 
Then, by Proposition~6.11 in \cite{PV6}, \eqref{gral fla b} holds if $\varphi(rad(b)) \mid a$,
where $rad(b)=p_1 \cdots p_r$ is the radical of an integer $b$ whose decomposition is $b=p_1^{a_1} \cdots p_r^{a_r}$.
\end{rem}

Applying the argument in the previous remark, by Corollaries 6.2--6.5  in \cite{PV6} we have the following general examples.

\begin{exam} \label{coros 6.3-6.5}
Let $p$ be an odd prime and $a\in \N$. If $R_{a}$ denotes any finite commutative local ring with identity whose residue field is of size $q=p^a$, then we have 	
	\begin{equation} \label{g2}
		\begin{aligned} 
			& g_{R_{2a}}(\tfrac{p^a+1}2) = 2,
		\end{aligned}	
	\end{equation}
	\begin{equation} \label{g=3}
		\begin{aligned}
			& g_{R_{3a}}(\tfrac{p^{2a}+p^a+1}3)=3, 		  & & \qquad \text{if \: $p \equiv 1 \!\! \pmod 3$}, \\[1mm]
			& g_{R_{6a}}(\tfrac{p^{4a}+p^{2a}+1}3)=3,      & & \qquad \text{if \: $p \equiv 2\!\! \pmod 3$}, \\[1mm]
		\end{aligned}
	\end{equation}
	\begin{equation} \label{g=5}
		\begin{aligned}
			& g_{R_{5a}}(\tfrac{p^{4a}+p^{3a}+p^{2a}+p^a+1}{5})=5, 			& & \qquad \text{if \: $p \equiv  1 \!\!\pmod 5$}, \\[1mm]
			& g_{R_{10a}}(\tfrac{p^{8a}+p^{6a}+p^{4a}+p^{2a}+1}{5})=5, 		& & \qquad \text{if \: $p \equiv  4 \!\!\pmod 5$}, \\[1mm]
			& g_{R_{20a}}(\tfrac{p^{16a}+p^{12a}+p^{8a}+p^{4a}+1}{5})=5, 	& & \qquad \text{if \: $p \equiv 2,3 \!\!\pmod  5$}, \\[1mm]
		\end{aligned}
	\end{equation}
	\begin{equation} \label{g=7}
		\begin{aligned}
			& g_{R_{7a}}(\tfrac{p^{6a}+p^{5a} + p^{4a} + p^{3a}+p^{2a}+p^a+1}{7})=7, & & \qquad \text{if \: $p \equiv  1 \!\!\pmod 7$}, \\[1mm]
			& g_{R_{14a}}(\tfrac{p^{12a}+p^{10a}+p^{8a} + p^{6a}+p^{4a}+p^{2a}+1}{7})=7, & & \qquad \text{if \: $p \equiv  6 \!\!\pmod 7$}, \\[1mm]
			& g_{R_{21a}}(\tfrac{p^{18a}+p^{15a}+p^{12a} + p^{9a}+p^{6a}+p^{3a}+1}{7})=7, & & \qquad \text{if \: $p \equiv  2 \!\!\pmod 7$}, \\[1mm]
			& g_{R_{42a}}(\tfrac{p^{36a}+p^{30a}+p^{24a} + p^{18}+p^{12a}+p^{6a}+1}{7})=7, & & \qquad \text{if \: $p \equiv  3,4,5 \!\!\pmod 7$}.
		\end{aligned}
	\end{equation}
\end{exam}


\section{Explicit values for $g_R(k)$ with $R$ local} \label{sec6}
In this last section we apply the results from the two previous ones, especially Theorem~\ref{teo waring Local case}, to get explicit Waring numbers over finite commutative local rings with identity. First, we adapt some known results for Waring numbers over finite fields to Waring numbers over finite commutative local rings $R$ to get 
$$2 \le g_R(k) \le 3,$$ 
and then consider the case when $(R,\frak m)$ is a local ring with prime residue field.

\subsection{Small values of $g_R(k)$ for $R$ local} 	
We first show that, under standard arithmetic conditions, if $g(k,q)=2$ then $g_R(k)=2$ or $3$ for any finite commutative local ring $R$ with residue field $\ff_q$.
\begin{thm} \label{gR23}
Let $(R,\frak m)$ be a finite commutative local ring with identity with residue field 
$\ff_{q}$ of prime characteristic $p$. 
Let $k \in \N$ with $(k,p)=1$ and put $k'=(k,q-1)$ such that $\frac{q-1}{k'} \dagger q-1$. 
If $g(k,q)=2$, then $g_R(k)=2$ if $R$ is a field and for $R$ not a field we have 
\begin{equation} \label{gRk23}
	g_{R}(k) =
		\begin{cases}
			2, & \qquad \text{if $q$ is even or else $q$ is odd and $k' \mid \tfrac{q-1}{2}$}, \\[1mm]
			3, & \qquad \text{if $q$ is odd and $k' \nmid \tfrac{q-1}{2}$.}
		\end{cases}
\end{equation}
\end{thm}

\begin{proof}
If $R$ is a field, then we have $g_R(k)=g(k,q)=2$, by \eqref{gR local}. 
Assume that $R$ is not a field. If $q$ is even or $q$ is odd and $k'\mid \tfrac{q-1}{2}$ then $G_R(k)$ is undirected by \eqref{nodirigido} and hence we have that $g_R(k)=2$, by \eqref{gR local}. 

On the other hand, if $q$ is odd and $k' \nmid \tfrac{q-1}{2}$ then $G_R(k)$ is directed and we are in the situation of $(b)$ in Theorem~\ref{teo waring Local case} with $g(k,q)=2$.
Since the maximal ideal of $R$ is not trivial
and we know that $3\le \gamma(\G(k,q)) \le p$, by Lemma \ref{lem 3gRWp}, then by \eqref{gR local dir} we have that 
$g_R(k)=2+1=3$, as we wanted to show.
\end{proof}

The previous result applies for all known cases in which $g(k,q)=2$ (see Section 6 in \cite{PV7}).
Putting together the results from Section 6 in \cite{PV7} with Theorem \ref{gR23} we will get some corollaries and examples.

\subsubsection*{Small values from Small's result}
We begin with Small's result to get small values of $g_R(k)$.

\begin{prop} \label{coro gR23}
	Let $(R,\frak m)$ be a finite commutative local ring with identity, which is not a field, with residue field 
	$\ff_{q}$ of prime characteristic $p$. 
	Let $k \in \N$ with $(k,p)=1$ such that $k\mid q-1$ and $\frac{q-1}{k} \dagger q-1$. If $2\le k \le \sqrt[4]{q}+1$ then 
	\begin{equation} \label{gRk23 fla}
	  g_{R}(k) =
		\begin{cases}
			2, & \qquad \text{if $q$ is even or else $q$ is odd and $k \mid \tfrac{q-1}{2}$}, \\[1mm]
			3, & \qquad \text{if $q$ is odd and $k \nmid \tfrac{q-1}{2}$.}
		\end{cases}
	\end{equation}
In particular, $g_R(2)=2$ if $q\equiv 1 \pmod 4$ and $g_R(2)=3$ if $q\equiv 3 \pmod 4$. 
Furthermore, we have that 
$g_R(2^r)=2$ if $q\equiv 1 \pmod{2^{r+1}}$ and $g_R(2^r)=3$ if $q \not \equiv 1 \pmod{2^{r+1}}$
for every $q\ge 2^{4r}$ with $r\ge 2$.
\end{prop}

\begin{proof}
Small's result \cite{Sm} says that $g(k,q)=2$ for every pair $k,q$ such that $2 \le k< \sqrt[4]q +1$ and $k\mid q-1$. 
This and Theorem \ref{gR23} directly imply the result in \eqref{gRk23 fla}. The remaining assertions are straightforward from \eqref{gRk23 fla}.
\end{proof}

\subsubsection*{Connected strongly regular graphs and semiprimitive pairs}
A strongly regular graph with parameters $v, \kappa, e, d$, denoted by $srg(v, \kappa, e, d)$, is a 
$\kappa$-regular graph (undirected by definition) with $v$ vertices such that for any pair of vertices $x,y$ the number of vertices adjacent (resp.\@ non-adjacent) to both $x$ and $y$ is $e\ge 0$ 
(resp.\@ $d\ge 0$). 
It is known that if $\G$ is a connected graph, then $\G$ is a strongly regular graph 
if and only if it has three distinct eigenvalues.

Let $q=p^m$ with $p$ prime and $m \in \N$.  If 
\begin{equation} \label{kqp}
	k\mid \tfrac{q-1}{p-1},
\end{equation}
the spectrum of $\G(k,q)$ is integral (\cite{PV2}, Theorem 2.1). In this case, there is a direct relationship between the spectra of GP-graphs $\G(k,q)$ and the weight distribution of the irreducible cyclic $p$-ary codes 
	$$\mathcal{C}(k,q) = \big\{ \big(\operatorname{Tr}_{q/p} (\gamma \omega ^{ki})\big)_{i=0}^{n-1} : \gamma \in \ff_q \big\}, $$ 
where $\omega$ is a primitive element of $\ff_q$ over $\ff_p$ and $n=\frac{q-1}{k}$. More precisely, the eigenvalues and  multiplicities of $\G(k,q)$ are in a 1-1 correspondence with the weights and frequencies of $\mathcal{C}(k,q)$ (see \S5 in \cite{PV2}).

In this way, connected strongly regular graphs of the form $\G(k,q)$ are spectrally related with $2$-weight irreducible cyclic codes of the form $\mathcal{C}(k,q)$. 
In \cite{SW}, Schmidt and White conjectured that all 2-weight irreducible cyclic codes over $\ff_p$
satisfying condition \eqref{kqp}
belong to one of the following disjoint families: ($a$) subfield subcodes, ($b$) semiprimitive codes, and ($c$) exceptional codes.
No other 2-weight irreducible cyclic codes beyond these were found, nor it is proved that they are all.
In particular, all these codes satisfy $p-1\mid n$.

Subfield subcodes correspond to the case when $k=\frac{p^m-1}{p^{a}-1}$ with $a<m$. 
In this case the graphs $\G(k,p^m)$ are not connected and hence we do not consider them. So, we restrict ourselves to semiprimitive codes and exceptional codes of the form $\mathcal{C}(k,q)$ which correspond to connected strongly regular graphs $\G(k,q)$, where $q=p^m$, which we now recall:

$\bullet$ \textit{Semiprimitive codes $\mathcal{C}(k,q)$} 
correspond to those $k>1$ such that $-1$ is a power of $p$ modulo $k$, that is $k \mid p^t+1$ for some $t$.
In this case, either 
$$\text{$k=2$, $p$ is odd and $m$ is even} \qquad \text{or else} \qquad \text{$k\mid p^\ell +1$, $\ell \mid m$ and $\tfrac{m}{\ell}$ is even}$$ 
(for some $\ell$). If in addition $k\ne p^{\frac m2}+1$ then we say that $(k,p^m)$ is a \textit{semiprimitive pair} of integers.

$\bullet$ \textit{Exceptional codes $\mathcal{C}(k,q)$} correspond to the 11 pairs below.\footnote{Note that in Table 1 in \cite{PV7} there is a typo, the pair $(35, 3^{13})$ should read $(35, 3^{12})$.} 
\renewcommand{\arraystretch}{1.2}
\begin{equation} \label{exc pairs}
\begin{tabular}{||c||c|c|c|c|c|c|c|c|c|c|c||} 
\hline 
$k$ 	& $11$ & $19$ & $35$ & $37$ & $43$ & $67$ & $107$ & $133$ & $163$ & $323$ & $499$ \\ 
\hline
$p^m$   & $3^5$ & $5^9$ & $3^{12}$ & $7^9$ & $11^7$ & $17^{33}$ & $3^{53}$ & $5^{18}$ & $41^{81}$ & $3^{144}$ & $5^{249}$ \\
\hline
\end{tabular}
\end{equation}
In this case we say that $(k,p^m)$ is an \textit{exceptional pair} of integers.

\begin{prop} \label{prop gR2}
Let $(R,\frak m)$ be a finite commutative local ring with identity with residue field 
$\ff_{q}$ of prime characteristic $p$. 
Let $k \in \N$ with $(k,p)=1$ such that $k\mid q-1$ and $\frac{q-1}{k} \dagger q-1$. If $\G(k,q)$ is a connected strongly regular graph then $g_R(k)=2$. In particular, if $(k,q)$ is a semiprimitive pair or an exceptional pair 
then $g_R(k)=2$.
\end{prop}

\begin{proof}
In \cite{PV6}, Lemma 3.1, we showed that the fact that $\G(k,q)$ is connected is equivalent to the fact that $g(k,q)$ exists which in turn is also equivalent to the fact that $\frac{q-1}k$ is a primitive divisor of $q-1$. Since $\G(k,q)$ is also strongly regular (hence in particular undirected), by Proposition~6.3 in \cite{PV7} we have that $g(k,q)=2$. 
Now, we want to apply Proposition~\ref{gR23}. 
Since $k\mid q-1$ we have that $k':=(k,q-1)=k$ and $\frac{q-1}{k'} \dagger q-1$. Also, since $\G(k,q)$ is undirected we have $k\mid \frac{q-1}2$ by ($a$) in Corollary \ref{-1RFq}. 
Hence, by Theorem~\ref{gR23}, we get $g_R(k)=2$.
The remaining assertions are obvious.
\end{proof}

\begin{exam}
By Example 6.5 in \cite{PV7} we have that $g(k,p^m)=2$ for every pair $(k,p^m)$ in \eqref{exc pairs}.
By applying Proposition \ref{gR23} we have that 
	$$g_{R_i}(k_i) = 2$$
where $(k_1,k_2,k_3,k_4,k_5,k_6,k_7,k_8,k_9,k_{10},k_{11}) = (11, 19, 35, 37,43,67,107,133, 163, 323, 499)$ 
and $R_1, \ldots, R_{11}$ are finite commutative local rings with identity, 
whose residue fields have corresponding sizes $3^5, 5^9, 3^{12}, 7^9, 11^7, 17^{33}, 3^{53}, 5^{18}, 41^{81}, 3^{144}$ and $5^{249}$. 
\hfill $\lozenge$
\end{exam}

We now give several instances of $g_R(k)=2$ by using explicit semiprimitive pairs $(k,q)$.
\begin{coro}
	Let $(R,\frak m)$ be a finite commutative local ring with identity, which is not a field, with residue field 
	$\ff_{q}$ of prime characteristic $p$.
	Let $\ell,s \in \N$.
	Then, we have the following:
	\begin{enumerate}[$(a)$]
		\item If $q=p^{2\ell}$ then $g_R(k)=2$ for $k=2,3,4$  and $p\equiv -1 \pmod k$, with $\ell\ge 2$ if $p=k-1$. \sk 
		
		\item If $q=p^{2\ell s}$ then $g_{R}(p^{\ell}+1)=2$ for $s\ge 2$. \sk
						
		\item If $q=p^{2\ell s}$ then $g_{R}(\tfrac{p^{\ell}+1}2)=2$ for any $p$ odd. \sk
			
		\item If $q=p^{2\ell s}$ then $g_{R}(\tfrac{p^{\ell}+1}h)=2$ for any $\ell$ odd and $p  \equiv -1 \pmod h$ with $h>2$. \sk
		
		\item If $q=p^{2\ell}$ then $g_{R}(\frac{p^{\ell}+1}{h})=2$ for any $p$ odd and $h\mid p^\ell+1$ with $h>1$.
	\end{enumerate}
\end{coro}

\begin{proof}
In the proofs of Corollary 6.6 and Proposition 6.7 in \cite{PV7} it is shown that the pairs $(k,q)$	appearing in the statement are semiprimitive pairs under the corresponding conditions. 
The result thus follows directly by Proposition \ref{prop gR2}.
\end{proof}

\subsection{The Waring number for local rings with prime residue field} 
We now study in more detail Waring numbers for local rings $R$ having their maximal ideal $\frak m$ of the maximum possible size, that is $|R|=p^s$ and $|\frak m|=p^{s-1}$ with $p$ prime, so that the associated residue field $R/\frak m$ is (isomorphic to) the prime field $\ff_p$. For instance, 
$$\Z_{p^s} \qquad \text{or} \qquad \Z_p[x]/(x^s)$$ are examples of such local rings, with maximal ideals $(p)=\{0,p,2p,3p,\ldots,(p^s-1)p\}$ and $(x)$, 
respectively.

In the case when $(R,\frak m)$ has prime residue field, the hypothesis of primitive divisibility in  
Theorem \ref{teo waring Local case} can be removed, and we have the following result.
\begin{prop} \label{prop Zps}
Let $R$ be a finite commutative local ring with identity of size $p^s$, with $p$ prime and $s\in \N$, having prime residue field $\ff_p$. Let $k \in \N$ be such that $(k,p)=1$ and put $k'=(k,p-1)$. 
Then, $g_{R}(k)$ always exists and we have the following cases:
	\begin{enumerate}[$(a)$]
		\item If $p=2$, then for any $k$ odd we have $g_{R}(k)=1$ if $s=1$ and $g_{R}(k)=2$ if $s\ge 2$. \sk 
		
		\item If $p$ is odd, then we have two cases: \sk 
			\begin{enumerate}[$(i)$]
				\item If $k'\mid \frac{p-1}{2}$ then 
				$$g_{R}(k)=
				\begin{cases}
				\hfil 1, & \qquad \text{if $k'=1$ and $s=1$}, \\[1mm]
				\hfil 2, & \qquad \text{if $k'=1$ and $s\ge 2$}, \\[1mm]
				\hfil g(k,p), & \qquad \text{if $k'\ge2$}. 
				\end{cases}$$
		
				\item If $k'\nmid \frac{p-1}{2}$ then 
				$$g_{R}(k) = 
				\begin{cases}
				\hfil g(k,p), 	& \qquad \text{if $\gamma(\G(k,p))\le g(k,p)$ or $s=1$},	\\[1mm]
				\hfil g(k,p)+1, & \qquad \text{if $\gamma(\G(k,p))> g(k,p)$ and $s\ge 2$}.  
				\end{cases} $$
 			\end{enumerate}
		\end{enumerate}
\end{prop}

\begin{proof}
The condition $\frac{p-1}{k'} \dagger p-1$ is automatic for $p$ prime and hence the result follows directly from Theorem \ref{teo waring Local case} using condition \eqref{nodirigido}, separating the cases of even and odd characteristic.
\end{proof}

As a direct consequence, we obtain the following exact Waring numbers.
\begin{coro} \label{coro exact Zp}
Let $R$ be a finite commutative local ring with identity of size $p^s$, with $p$ an odd prime and $s\in \N$, having prime residue field $\ff_p$ --for instance, $\Z_{p^s}$ or $\Z_p[x]/(x^s)$--. 
Hence, we have that 
	\begin{equation} \label{gR2}
		g_{R}(2) =	\begin{cases}
		\hfil 2, & \qquad \text{if $p\equiv 1 \pmod{4}$},	\\[1mm]
		\hfil 3, & \qquad \text{if $p\equiv 3 \pmod{4}$},  
		\end{cases}
	\end{equation}
and also $g_{R}(\tfrac{p-1}{2}) = \tfrac{p-1}{2}$ for $p \ge 5$ and $g_{R}(p-1)=p$.	
\end{coro}
	
\begin{proof}
First notice that since $p$ is odd then $2$, $\frac{p-1}2$ and $p-1$ are all coprime with $p$. Also, it is a classic result of Cauchy (\cite{Ca}) that if $p$ is prime, then 
\begin{equation} \label{Cauchy}
	g(k,p) \le k
\end{equation}
with equality if $k=1,2,\tfrac{p-1}{2}$ and $p-1$, that is $g(1,p)=1$ which is obvious and 
	$$g(2,p)=2, \qquad g(\tfrac{p-1}2,p) = \tfrac{p-1}2, \qquad \text{and} \qquad g(p-1,p)=p-1.$$
		
Now, for $k=\frac{p-1}2$ we have $k'=(k,p-1)=\frac{p-1}2$ and since $k' \mid \frac{p-1}{2}$
the graph $G_R(\frac{p-1}2)$ is non-directed and hence, item ($i$) in part $(b)$ of Proposition \ref{prop Zps} implies that 
$g_{R}(\tfrac{p-1}{2})=\tfrac{p-1}{2}$ for any $p\ge 5$, since $k'=1$ if and only if $p=3$.
		
For $k=p-1$, we have $k'=(k,p-1)=p-1$ and since $k' \nmid \frac{p-1}2$ the graph $G_R(p-1)$ is directed and we are in the situation of $(ii)$ in item $(b)$ of Proposition \ref{prop Zps}. Since the graph $\G(p-1,p)=C_p$ has girth $p$, which is lower than $g(p-1,p)=p-1$, then we have that 
	$$g_R(p-1)=g(p-1,p)+1=p.$$
		
Finally, for $k=2$ we have $k'=(2,p-1)=2$. Now, $2\mid \frac{p-1}{2}$ if and only if $p\equiv 1\pmod{4}$. Also, when 
$p\equiv 3\pmod{4}$ we have that $\G(k,p)$ is the classic Paley graph which is known to have diameter 2 and girth $3$. 
Therefore, by item $(b)$ of Proposition \ref{prop Zps} we obtain that $g_{R}(2)=2$ if $p\equiv 1 \pmod{4}$ --case $(i)$-- and $g_{R}(2)=3$ if $p\equiv 3 \pmod{4}$ --case $(ii)$--, and hence the result is proved.
\end{proof}

\begin{rem}
There are some known upper bounds for Waring numbers over finite fields of prime order. If $p$ is a prime, besides \eqref{Cauchy} we have the bounds
$g(k,p)\le [\frac k2]+1$ for $k<\tfrac{p-1}{2}$ and $g(k,p) \le 83\sqrt{k}$ (see \cite{CMS} and \cite{CP} respectively).
By Proposition \ref{prop Zps} we have that 
	 $$g_R(k) \le [\tfrac k2]+2 \qquad \text{for $k<\tfrac{p-1}{2}$} \quad \qquad \text{and} \quad \qquad g_R(k) \le 83\sqrt{k}+1$$
for any local ring $R$ having prime residue field $\ff_p$.
\end{rem}

We now show that in every finite commutative local ring $R$ of order $p^2$ with $p$ an odd prime, every non-zero element can be written as the sum of two or three $\frac{p+1}2$-th powers, depending on the congruence modulo 4 of $p$.
\begin{coro} \label{coro p2}
Let $R$ be a finite commutative local ring with identity of order $p^2$, with $p$ an odd prime. Then, $g_R(\frac{p+1}2)=2$ if $R$ is a field and if $R$ is not a field we have that
	$$ g_{R}(\tfrac{p+1}{2})=
	\begin{cases}
		2, & \qquad \text{if $p\equiv 1 \pmod{4}$,} \\[1mm]
		3, & \qquad \text{if $p\equiv 3\pmod{4}$.}
	\end{cases}$$ 
\end{coro}

\begin{proof}
It is well-known that for any prime $p$ there are only three (up to ring isomorphism) finite commutative local rings with identity of order $p^2$: the field $\ff_{p^2}$, and the rings with zero-divisors 
$$R_1=\Z_p[x]/(x^2) \qquad \text{and} \qquad R_2=\Z_{p^2}$$ 
with maximal ideals isomorphic to $\Z_p$. In the first case the residue field is $\ff_{p^2}$ while for $R_i$ it is $R_i/\Z_p \simeq \ff_p$ with $i=1,2$.
	
First, we notice that for $k=\frac{p+1}2$ we have $(k,p)=1$.
Also, it is known that for $p$ odd we have 
\begin{equation*} \label{g=2}
	g(\tfrac{p+1}2,p^2)=2
\end{equation*}
(take $a=1$ in Corollary 6.2 in \cite{PV6}) and so $g_{R}(\frac{p+1}{2})=g(\frac{p+1}2,p^2)=2$ when $R=\ff_{p^2}$. 
	
Now, assume that $R$ is not a field. Then, its residual field is the prime field $\ff_p$, for both $R_1$ and $R_2$.
It is easy to see that if $p$ is an odd prime then
	$$k':=(k,p-1)=(\tfrac{p+1}2,p-1)=
		\begin{cases}
		1, & \quad \text{ if $p\equiv 1\pmod{4}$}, \\[1mm]
		2, & \quad \text{ if $p\equiv 3 \pmod{4}$}.
	\end{cases}$$ 
If $p\equiv 1 \pmod 4$ then $k'=1$ divides $\frac{p+1}2$ and thus by \eqref{gR local} we have that 
	$$g_R(\tfrac{p+1}2)=2.$$
On the other hand, if $p\equiv 3 \pmod 4$ then $k'=2$ does not divide $\frac{p+1}2$ which is odd, and hence we are in the situation of $(b)$ in Theorem \ref{teo waring Local case}. Since $k'=2$, we have that $g(k,p)=g(k',p)=g(2,p)=2$. 
Since the maximal ideal of $R$ is not trivial and $3\le \gamma(\G(k,p)) \le p$ by Lemma \ref{lem 3gRWp}, we have that 
	$$g_R(\tfrac{p+1}2) = g(\tfrac{p+1}2,p)+1=3$$ 
by \eqref{gR local dir}, as we wanted to show.
\end{proof}

\goodbreak 

For instance, every element of $\Z_{49}$ or $\Z_7[x]/(x^2)$ can be written as the sum of three 4-th powers, every element of $\Z_{121}$ or $\Z_{11}[x]/(x^2)$ can be written as the sum of three 6-th powers and every element of $\Z_{169}$ or $\Z_{13}[x]/(x^2)$ can be written as the sum of two 7-th powers.

We conclude with an example illustrating the last two corollaries.
\begin{exam}
Let $R$ be a finite commutative local ring of order $p^2=25$ which is not a field, hence $R=\Z_{25}$ or $R=\Z_5[x]/(x^2)$.
By Corollaries \ref{coro exact Zp} and \ref{coro p2} we have that 
$$g_R(2)=g_R(3)=2 \qquad \text{and} \qquad g_R(4)=5.$$ 
In fact, for $R=\Z_{25}$ we have 
$$ \begin{tabular}{|c|ccc|}
	\hline
  $a$ & $a^2$ & $a^3$ & $a^4$ \\ \hline
	0 & 0 & 0 & 0 \\
	1 & 1 & 1 & 1 \\
	2 & 4 & 8 & 16 \\
	3 & 9 & 2 & 6 \\
	4 & 16& 14& 6 \\ \hline
\end{tabular}
\: \: \:
\begin{tabular}{|c|ccc|}
\hline
$a$ & $a^2$ & $a^3$ & $a^4$ \\ \hline
	5 & 0 & 0 & 0 \\
	6 & 11& 16& 21 \\
	7 & 24& 18& 1 \\
	8 & 14& 12& 21 \\
	9 & 6 & 4 & 11 \\	\hline	
\end{tabular} 
\: \: \:
\begin{tabular}{|c|ccc|}
\hline
$a$ & $a^2$ & $a^3$ & $a^4$ \\ \hline
10& 0 & 0 & 0 \\ 
11& 21& 6 & 16 \\ 	
12& 19& 3 & 11 \\ 	
13& 19& 22& 11 \\ 
14& 21& 19& 16 \\	\hline	
\end{tabular}
\: \: \:
\begin{tabular}{|c|ccc|}
\hline
$a$ & $a^2$ & $a^3$ & $a^4$ \\ \hline
15 & 0 & 0  & 0 \\
16 & 6 & 21 & 11 \\
17 & 14& 13 & 21 \\
18 & 24& 7  & 1 \\
19 & 11& 9  & 21 \\\hline	
\end{tabular} 
\: \: \:
\begin{tabular}{|c|ccc|}
\hline
$a$ & $a^2$ & $a^3$ & $a^4$ \\ \hline
20 & 0 & 0 & 0 \\
21 & 16& 11& 6 \\
22 & 9 & 23& 6 \\
23 & 4 & 17& 16 \\
24 & 1 & 24& 1 \\ \hline	
\end{tabular}$$
from which it is clear that any element of $\Z_{25}$ can be written as the sum of two squares or two cubes or five fourth powers, but not all the elements are a square or a cube or can be written as sums of less than 5 fourth powers. For instance, in $\Z_{25}$ we have
$$15=1^2+8^2= 1^3+4^3=1^4+7^4+9^4+18^4+24^4$$
but $15$ is not a square or a cube or a fourth power in $\Z_{25}$ and cannot be written as sum of 2, 3 or 4 fourth powers (check!). 
Similarly, one can check that for every $ax+b \in \Z_5[x]/(x^2) \smallsetminus \{0\}$ there exist $f_1(x), f_2(x), g_1(x), g_2(x), h_1(x), h_2(x), h_3(x), h_4(x), h_5(x) \in \Z_5[x]/(x^2)$ such that 
$$ax+b =f_1(x)^2+f_2(x)^2 = g_1(x)^3+g_2(x)^3 = h_1(x)^4+h_2(x)^4+h_3(x)^4+h_4(x)^4+h_5(x)^4.$$
We omit the details of these computations. 
\hfill $\lozenge$
\end{exam}

\end{document}